\documentclass[11pt,twoside]{article}
\usepackage{amsmath,amssymb,amsthm,mathptmx,hyperref,xcolor}
\usepackage[english]{babel}%
\usepackage[inner=3.4cm,outer=3.65cm,bottom=3cm]{geometry}

\usepackage{mathtools}
\usepackage{marginnote}

\theoremstyle{remark}
\newtheorem{remark}{Remark}[section]
\theoremstyle{definition}
\newtheorem{theorem}{Theorem}[section]
\newtheorem{definition}[theorem]{Definition}

\newtheorem{proposition}[theorem]{Proposition}
\newtheorem{lemma}[theorem]{Lemma}
\newtheorem{corollary}[theorem]{Corollary}
\newtheorem{hypothesis}[theorem]{Hypothesis}

\usepackage{cancel}
\DeclareMathOperator{\R}{\mathbb{R}}

\newcommand{\C}{C}
\DeclareMathOperator{\tS}{\tilde{\mathbb{S}}}
\DeclareMathOperator{\Se}{\mathbb{S}}

\DeclareMathOperator{\N}{\mathbb{N}}

\DeclareMathOperator{\ra}{\rightarrow}
\newcommand{\de}{\,\text{d}}

\newcommand{\sym}{\text{sym}}
\newcommand{\skw}{\text{skw}}
\newcommand{\loc}{\text{loc}}
\newcommand{\ddd}{\text{d}}

\newcommand{\f}[1]{{\pmb{ #1}}}%%%%{{\mathbf{ #1}}}
\DeclareMathOperator{\di}{\nabla\cdot}

\newcommand{\tu}{\tilde{\f u}}

\newcommand{\ve}{\varepsilon}

\renewcommand{\t}{\partial_t}

\newcommand{\sy}[1]{(\nabla \f #1)_{{\sym}}}

\newcommand{\vv}{\tilde{\f v}}

\newcommand{\syv}{(\nabla \tilde{\f v})_{{\sym}}}

\newcommand{\Xb}{\mathfrak{X}}

\newcommand{\Yb}{\mathfrak{Z}}
\newcommand{\Xbzero}{\mathfrak{X}_0}
\newcommand{\Ybzero}{\mathfrak{Z}_0}

\newcommand{\vgam}{\f v_{\gamma}}
\newcommand{\Sgam}{\Se_{\gamma}}

\newcommand{\dreidots}{\text{\,\multiput(0,-2)(0,2){3}{$\cdot$}}\,\,\,\,}

\newcommand{\e}{\text{\textit{e}}}

\newcommand{\Om}{\Omega}
\newcommand{\OmT}{\Om\times(0,T)}

\newcommand{\pOmT}{\partial\Om\times(0,T)}
\newcommand{\OmTzero}{\Omega\times[0,T)}

\newcommand{\nvec}{\pmb{n}}
\newcommand{\grad}{\nabla}
\DeclareMathOperator{\Div}{\nabla \cdot}

\newcommand{\pt}{\partial_t}
\newcommand{\D}{{\mathrm D}}
\newcommand{\dx}{\,{\mathrm d}x}
\newcommand{\dtau}{\,{\mathrm d}\tau}
\newcommand{\dt}{\,{\mathrm d}t}
\newcommand{\ds}{\,{\mathrm d}s}

\newcommand{\jaumannder}[1]{\overset{\triangledown}{#1}}
\newcommand{\symmgrad}[1]{\np{\grad #1}_{\text{sym}}}
\newcommand{\skewgrad}[1]{\np{\grad #1}_{\text{skw}}}
\newcommand{\ddt}{\frac{\mathrm{d}}{\mathrm{d}t}}
\newcommand{\transpose}{\top}
\newcommand{\idmatrix}{\mathbb{I}}
\DeclareMathOperator{\trace}{Tr}
\newcommand{\zerotens}{\mathbb{O}}
\newcommand{\zerovec}{\pmb{0}}
\newcommand{\vvel}{\pmb{v}}
\newcommand{\tvvel}{{\tilde\vvel}}
\newcommand{\vpres}{p}
\newcommand{\vf}{\pmb{f}}

\newcommand{\Stens}{\mathbb{S}}
\newcommand{\tStens}{\tilde{\mathbb{S}}}
\newcommand{\Ttens}{\mathbb{T}}
\newcommand{\tTtens}{\tilde{\mathbb{T}}}
\newcommand{\potential}{\mathcal{P}}
\newcommand{\norm}[1]{\lVert#1\rVert}

\newcommand{\snorm}[1]{{\lvert #1 \rvert}}

\newcommand{\snormL}[1]{{\Bigl\lvert #1 \Big\rvert}}

\newcommand{\LR}[1]{L^{#1}}
\newcommand{\LRloc}[1]{L^{#1}_{\mathrm{loc}}} 
\newcommand{\CR}[1]{C^{#1}}  
\newcommand{\CRi}{\CR \infty}
\newcommand{\CRci}{\CR \infty_0}

\newcommand{\WSR}[2]{W^{#1,#2}} 
 
\newcommand{\WSRloc}[2]{W^{#1,#2}_{\mathrm{loc}}} 
\newcommand{\HSR}[1]{H^{#1}} 
\newcommand{\HSRN}[1]{H^{#1}_0} 
\newcommand{\HSRloc}[1]{H^{#1}_{\mathrm{loc}}}

\newcommand{\LRsigma}[1]{L^{#1}_{\sigma}} 
\newcommand{\HSRNsigma}[1]{H^{#1}_{0,\sigma}} 
\newcommand{\CRcisigma}{\CR{\infty}_{0,\sigma}}
\newcommand{\LHT}{\mathrm{LH}_T}

\newcommand{\LRdev}[1]{L^{#1}_{\delta}} 
\newcommand{\HSRdev}[1]{H^{#1}_{\delta}} 

\newcommand{\XT}{X_T}

\newcommand{\YT}{Y_T}

\newcommand{\ZT}{Z_T}

\newcommand{\np}[1]{(#1)}

\newcommand{\bp}[1]{\big(#1\big)}
\newcommand{\bb}[1]{\big[#1\big]}
\newcommand{\Bp}[1]{\bigg(#1\bigg)}

\newcommand{\set}[1]{\ensuremath{\{#1\}}}

\newcommand{\setc}[2]{\ensuremath{\{#1 : #2\}}}
\newcommand{\setcl}[2]{\ensuremath{\bigl\{#1 : #2\bigr\}}}
\newcommand{\setcL}[2]{\ensuremath{\biggl\{#1 : #2\biggr\}}}
\newcommand{\closure}[2]{\overline{#1}^{#2}}
\renewcommand{\restriction}[2]{#1 | _{#2}}
\newcommand{\tin}{\text{in }}
\newcommand{\ton}{\text{on }}
\newcommand{\EEE}{\color{black}}

 \newcommand{\statespace}{Q}
\newcommand{\energy}{\mathcal{E}}
\newcommand{\timeES}{\LRloc{\infty}([0,T);\statespace)}
\newcommand{\cten}{C_{\text{\hypersetup{hidelinks}\ref{eq:150}}}}
\newcommand{\ctv}{C_{\text{\hypersetup{hidelinks}\ref{eq:apri.s}}}}

\newcommand{\PropEnIneq}{Prop.~3.3} % revised version
\newcommand{\ThmMain}{Theorem~3.4}

\author{Thomas Eiter\footnotemark[1] \and Katharina Hopf\footnotemark[1] \and
Robert Lasarzik%
\footnote{Weierstrass Institute for Applied Analysis and Stochastics,
Mohrenstr. 39, 10117 Berlin, Germany,
\newline
Emails: 
\texttt{thomas.eiter@wias-berlin.de}
\newline
\textcolor{white}{Emails:}
\texttt{katharina.hopf@wias-berlin.de}
\newline
\textcolor{white}{Emails:}
\texttt{robert.lasarzik@wias-berlin.de}
}
}
\date{\today}
\title%
{
	Weak-strong uniqueness and energy-variational solutions 
	for a class of 
	viscoelastoplastic fluid models 
}	

\begin{document}
%%%%%%%%%%%%%%%%%%%%%%%%%%%%%%%%%%%%%%%%%
%%%%%%%%%%%%%%%%%%%%%%%%%%%%%%%%%%%%%%%%%
\maketitle
%%%%%%%%%%%%%%%%%%%%%%%%%%%%%%%%%%%%%%%%%
\begin{abstract}
We study a model for a fluid showing viscoelastic and viscoplastic behavior,
which describes the flow in terms of the fluid velocity and an internal stress.
This stress tensor is transported via the Zaremba--Jaumann rate,
and it is subject to two dissipation processes: 
one induced by a nonsmooth convex potential and one by stress diffusion.
We show short-time existence of strong solutions
as well as their uniqueness
in a class of Leray--Hopf type weak solutions satisfying the tensorial component 
in the sense of an evolutionary variational inequality. 
The global-in-time existence of such generalized solutions has been established in a previous work.
We further study the 
limit when stress diffusion vanishes.
In this case, the above notion of generalized solutions is no longer suitable,
and we introduce the concept of energy-variational solutions, 
which is based on an inequality for the relative energy.
We derive general properties of energy-variational solutions 
and show their existence
by passing to the non-diffusive limit
in the relative energy inequality satisfied by generalized solutions
for non-zero stress diffusion.
\end{abstract}

\section{Introduction}

In a three-dimensional bounded domain $\Omega\subset\R^3$,
we consider the flow of an incompressible viscoelastoplastic fluid 
governed by the equations
\begin{subequations}\label{eqvis}
\begin{align}
\t \f v + ( \f v \cdot \nabla ) \f v  - \di \left ( \eta\mathbb{S}+ 2\mu (\nabla \f v )_{\sym} \right ) + \nabla p = \f f,
\quad \di \f v &= 0 
&& \tin\OmT,\label{eqvis1}\\
\t \Se + (\f v \cdot \nabla) \Se + \Se (\nabla \f v )_{\skw} -(\nabla \f v)_{\skw} \Se+ \partial \mathcal{P}(\Se)-\gamma \Delta \Se &\ni \eta(\nabla \f v)_{\sym}  && \tin\OmT,\label{eqvis2}
\\
\vvel=0, \qquad \gamma\,\nvec\cdot\grad\Stens&=0 && \ton \pOmT,
\label{eq:bdrycond}
\\
\vvel(\cdot,0)=\vvel_0,\qquad\Stens(\cdot,0)&=\Stens_0 &&\tin \Omega,
\label{eq:initcond}
\end{align}
\end{subequations}
where $(0,T)$ is a time interval
with $T\in(0,\infty]$.
Equation \eqref{eqvis1}
describes the evolution of the fluid flow 
(with constant density $\rho=1$) 
in terms of the Eulerian velocity $\vvel\colon\OmT\to\R^3$ and pressure $\vpres\colon\OmT\to\R$,
subject to an external forcing $\vf\colon\OmT\to\R^3$.
The fluid stress $\Ttens=\eta\mathbb{S}+ 2\mu (\nabla \f v )_{\sym}-p\idmatrix$
decomposes into the classical term $2\mu (\nabla \f v )_{\sym}-p\idmatrix$ for Newtonian fluids 
and an additional stress $\Stens\colon\OmT\to\R^{3\times3}$.
Here $(\nabla \f v )_{\sym}=\frac{1}{2}\np{\grad\vvel+\grad\vvel^\transpose}$ is the rate-of-strain tensor,
and $\mu>0,\,\eta\ge0$ denote fixed constants.
Since $\vpres$ shall describe the physical pressure, that is, 
the (negative of the) spherical part of the fluid stress,
we assume $\Stens$ to be a symmetric deviatoric tensor field, that is,
$\Stens=\Stens^\transpose$ and $\trace\Stens=0$.
By equation \eqref{eqvis2}, the extra stress $\Stens$
is transported via the Zaremba--Jaumann rate
\begin{equation}\label{eq:jaumannder}
\jaumannder\Stens
\coloneqq
\t \Se + (\f v \cdot \nabla) \Se + \Se (\nabla \f v )_{\skw} -(\nabla \f v)_{\skw} \Se,
\end{equation}
where $(\nabla \f v )_{\skw}=\frac{1}{2}\np{\grad\vvel-\grad\vvel^\transpose}$.
Moreover, $\Stens$ is subject to a diffusion process induced by the term $\gamma\Delta\Stens$ with $\gamma\geq 0$
and an additional nonlinear dissipation due to the subdifferential  $\partial\potential(\Stens)$
of the nonsmooth convex potential $\potential$.
The system is completed 
by no-slip conditions for $\vvel$ and 
homogeneous Neumann conditions for $\Stens$ (in the case $\gamma>0$) on the boundary
as well as initial conditions.

The Zaremba--Jaumann derivative \eqref{eq:jaumannder} describes only one possible choice
for the objective stress rate,
and there exist other objective tensor derivatives
used for the description of stress evolution.
However, as we explain below, the model \eqref{eqvis} is motivated from 
geodynamics, where the Zaremba--Jaumann derivative is a common choice 
(see \cite{MoDuMu02MCMV,GerYue07RCMM,HerrenGeryaVand2017IRSDF,PHGAV2019SAFG}).
Moreover, it is crucial in the mathematical analysis of \eqref{eqvis}
since it (formally) guarantees the identity
\[
\ddt\int_\Omega\frac{1}{2}\snorm{\Stens}^2\dx
=\int_\Omega\pt\Stens:\Stens\dx
=\int_\Omega\jaumannder\Stens:\Stens\dx.
\]
This property can be used to reveal information on the evolution of the total quadratic energy
\[
	\energy(t):=\int_\Omega\Bp{\frac{1}{2}\snorm{\vvel(x,t)}^2
+\frac{1}{2}\snorm{\Stens(x,t)}^2}\dx,
\]
which consists of the kinetic energy associated with $\vvel$ 
and the stored elastic energy associated with $\Stens$.
More precisely, smooth solutions $\np{\vvel,\Stens}$ to \eqref{eqvis}
formally satisfy the energy-dissipation balance
\begin{equation}\label{eq:enbalance}
\energy(t)
+\int_0^t\int_\Omega\bp{\mu\snorm{\symmgrad{\vvel}}^2
+\partial\potential(\Stens):\Stens
+\gamma\snorm{\grad\Stens}^2}\dx\dtau
=\energy(0)
+\int_0^t\langle\f f,\vvel\rangle\dtau
\end{equation}
for all $t\in[0,T)$.
This shows that the total energy
is dissipated by three processes: 
the direct fluid viscosity with parameter $\mu>0$,
the nonsmooth stress-dissipation potential,
and the stress diffusion with parameter $\gamma\geq 0$.

The main non-standard feature of the system \eqref{eqvis} is
the occurrence of the set-valued subdifferential $\partial\potential(\Stens)$ in \eqref{eqvis2}, 
the meaning of which will be specified in the following.
Observe that the energy-dissipation balance \eqref{eq:enbalance} 
suggests to examine \eqref{eqvis} in an $\LR{2}$ framework and we
thus seek stress tensors $\Se=\Se(t)$ taking values in the space
$\LRdev{2}(\Omega)
=\setc{\Stens\in\LR{2}(\Omega)^{3\times3}}
{\Stens=\Stens^\transpose,\ \trace\Stens=0}$.
The dissipation potential $\potential$ is now defined to be a convex and lower semicontinuous function
$\potential\colon\LRdev{2}(\Omega)\to[0,\infty]$ that satisfies $\potential(\zerotens)=0$, where $\zerotens$ denotes the zero tensor.
By definition, the convex subdifferential $\partial\potential$ of $\potential$ is then given by 
\[
\Stens\mapsto\partial\potential(\Stens)
\coloneqq\setcL{\mathbb{G}\in\LRdev{2}(\Omega)}{\potential(\tStens)\geq\potential(\Stens)+\int_\Omega\mathbb{G}:(\tStens-\Stens)\dx \ \text{ for all }\tStens\in\LRdev{2}(\Omega)}.
\]
Such nonsmooth dissipation potentials allow to include plastic effects in the model, and related
viscoelastoplastic fluid models 
are used in geodynamics to describe rock deformation in the lithosphere;
see \cite{MoDuMu02MCMV,GerYue07RCMM},
where 
\begin{equation}
\potential(\Stens)=
\int_\Omega\mathfrak P(\Stens(x))\dx,
\qquad
\mathfrak P(\Stens)=
\begin{cases}
\frac{a}{2}\snorm{\Stens}^2 &\text{if } \snorm{\Stens}\leq\sigma_{\text{yield}},
\\
\infty &\text{if }\snorm{\Stens}>\sigma_{\text{yield}}.
\end{cases}\label{Pexample}
\end{equation}
Here $a>0$ is a constant, 
and the yield stress $\sigma_{\text{yield}}>0$ determines the transition to plastic behavior.
One readily verifies that $\potential$ defined in this way has the above properties.
Further examples for possible choices of $\potential$ can be found in \cite{EiterHopfMielke2021}.

The beginning of the mathematical analysis of viscoelastic fluid models,
also using objective derivatives different from the Zaremba--Jaumann rate \eqref{eq:jaumannder},
can be dated back to the middle 1980s; 
see \cite{JoReSa85HCTF,RenRen86LSPC,ReHrNo87MPV,CooSch91ILFV,Rena00MAVF} for example.
Since all objective derivatives come along with strong nonlinearities,
the first result on global existence of weak solutions was only established several years later 
by Lions and Masmoudi \cite{LM_2000}, 
who studied the system \eqref{eqvis}
for $\gamma=0$ 
and a quadratic dissipation potential such that $\partial\potential(\Stens)=a\Stens$ for some $a\geq0$.
In this case, \eqref{eqvis2} becomes a transport equation,
and existence can be deduced from the propagation of compactness in $L^2$.
This tool is no longer available when $\potential$ is nonlinear and nonsmooth.
For such potentials, large-data global existence can be achieved for diffusive regularizations of the tensorial transport equation, as recently demonstrated in~\cite{BMPS2018PDEA,BaBuMa21LDET,EiterHopfMielke2021}.
The article~\cite{EiterHopfMielke2021} considers problem~\eqref{eqvis} with $\gamma>0$ and proves  global existence of generalized solutions compsed of a weak formulation for~\eqref{eqvis1} and a variational inequality for~\eqref{eqvis2} (cf.~Definition \ref{def:gensol} below).
In the present article, we continue this analysis in two directions.

In a first part, we investigate the case $\gamma>0$, complementing and refining the existence analysis of generalized solutions in~\cite{EiterHopfMielke2021}.
Here, our main results are the short-time existence of strong solutions (see Theorem \ref{thm:strongsol}) as well their uniqueness among  generalized solutions  (see Theorem \ref{thm:weakstrong}).
The presence of the nonsmooth dissipation potential $\potential$ renders the construction of (local) strong solutions a non-trivial and interesting question.
Indeed, some care has to be taken to derive \textit{a priori} estimates compatible with the nonsmoothness of $\potential$, and our construction strongly relies on the fact that $\gamma>0$.
Concerning our uniqueness result, we note that the generalized solutions as considered here comprise the family of Leray--Hopf weak solutions of the Navier--Stokes equations (for $\eta=0$),
and hence the uniqueness of generalized solutions seems to be out of reach. 
The main step in the proof of the weak-strong uniqueness principle mentioned above is 
the derivation of an evolutionary inequality for the relative energy
\[
	\mathcal{R}(\f v , \Se | \vv, \tS) =
	\frac{1}{2} \| \f  v- \vv \|_{L^2(\Omega)}^2
	+ \frac{1}{2} \| \Se - \tS \|_{L^2(\Omega)}^2 \,
\] 
of the form
\begin{equation}
	\begin{aligned}
	\mathcal{R} (\f v(t) , \Se(t)| \vv(t),\tS(t)) 
	+\int_0^t   \bp{\mathcal{W}^{(\mathcal K)}(\f v , \Se| \vv,\tS ) 
	&+\mathcal{F}(\f v , \Se; \vv,\tS ) } \,
	\e^{\int_s^t \mathcal{K}(\vv,\tS) \de \tau} \de s
	\\  
	&
	\leq \mathcal{R} (\f v_0 , \Se_0| \vv(0),\tS(0))  \e^{\int_0^t \mathcal{K}(\vv,\tS) \de s }  \,
	\end{aligned}
	\label{relen.intro}
	\end{equation}
for $t\in(0,T)$,
see \eqref{relen} below,
which allows to compare generalized solutions $\np{\vvel,\Stens}$
with (more regular) test functions $\np{\tvvel,\tStens}$.
Here $\mathcal F$ contains more or less the (nonlinear) differential operator 
associated with problem \eqref{eqvis} 
applied to the test function $\np{\tvvel,\tStens}$,
and $\mathcal W^{(\mathcal K)}$ consists of terms describing the relative dissipation as well as 
terms arising from the nonlinearities in \eqref{eqvis}.
It further depends on the non-negative function $\mathcal K$,
which we call the regularity weight
since it determines the class of admissible test functions $\np{\tvvel,\tStens}$ 
such that \eqref{relen.intro} is meaningful.

 In the second part we study the problem without stress diffusion, that is, when $\gamma=0$.
Our interest in this case stems, among other things, from the original models used in geodynamics \cite{MoDuMu02MCMV,GerYue07RCMM}, where stress diffusion does not appear.
If $\gamma=0$, problem~\eqref{eqvis} reduces to the system
\begin{subequations}\label{eqvis.0}
\begin{align}
\t \f v + ( \f v \cdot \nabla ) \f v - \di \left ( \eta\mathbb{S}+ 2\mu (\nabla \f v )_{\sym} \right ) + \nabla p = \f f,
\quad \di \f v &= 0 
&& \tin\OmT,\label{eqvis1.0}
\\
\t \Se + (\f v \cdot \nabla) \Se + \Se (\nabla \f v )_{\skw} -(\nabla \f v)_{\skw} \Se+ \partial \mathcal{P}(\Se) &\ni \eta(\nabla \f v)_{\sym}  && \tin\OmT,\label{eqvis2.0}
\\
\vvel & =0  && \ton \pOmT,
\label{eq:bdrycond.0}
\\
\vvel(\cdot,0)=\vvel_0,\qquad\Stens(\cdot,0)&=\Stens_0 
&&\tin \Omega.
\label{eq:initcond.0}
\end{align}
\end{subequations}
As mentioned above, 
stress diffusion serves as a regularization making existence results accessible
by means of parabolic theory. 
When $\gamma=0$, the energy estimate \eqref{eq:enbalance} no longer controls the gradient of $\Se$ in $L^2_\loc(\overline\Om\times[0,T))$, and it becomes unclear how to pass to the limit along approximate solutions in the term $\big(\Se (\nabla \f v )_{\skw}{-}(\nabla \f v)_{\skw} \Se\big)$ in~\eqref{eqvis2.0} associated with the Zaremba--Jaumann derivative.
In this article, we provide a framework allowing us to treat the case $\gamma=0$.
It relies on the rather weak notion of energy-variational solutions to \eqref{eqvis.0}
(cf.~Def.~\ref{def:envar} below),
which is based on an inequality for the relative energy $\mathcal{R}(\f v , \Se | \vv, \tS)$,
adapting
the aforementioned relative energy inequality 
\eqref{relen.intro}
for generalized solutions when $\gamma>0$.
Under a convexity assumption, we may pass to the limit $\gamma\to0$
in this inequality using weak lower semicontinuity arguments.
This allows us to construct an energy-variational solution to \eqref{eqvis.0}
as the limit of a sequence of generalized solutions $\np{\vvel_\gamma,\Stens_\gamma}$
to \eqref{eqvis} for $\gamma>0$, see Theorem \ref{thm:envar}.

The idea to base a solution concept on a relative energy estimate goes back to
Lions \cite[Def.~4.1]{Lions_incompressible_1996}, 
who introduced
the notion of dissipative solutions
for the incompressible  Euler equations. 
A dissipative solution does not fulfill the differential equation in a distributional sense, but rather satisfies a relative energy inequality with respect to any sufficiently smooth test function. 
After the seminal work by Lions, this concept has been adopted in other contexts as well,
\textit{e.g.}, 
in the context of viscous electro-magneto-hydrodynamics~\cite{saintraymond}, liquid crystals~\cite{masswertig}, and nematic electrolytes~\cite{nematicelectro}.

Another generalized solution concept, which is often used in the context of fluid dynamics, 
is the notion of so-called measure-valued solutions~\cite{DiPernaMajda}. 
Measure-valued solutions carry more information than dissipative solutions, but this is achieved by increasing the degrees of freedom: In every point in time and space the solution carries an infinite dimensional measure. 
 The expectation of every measure-valued solution fulfills the dissipative formulation~\cite{weakstrongeuler}, which is identified as a desirable quantity in the case of liquid crystals~\cite{diss}.  
 Moreover, the concept of dissipative solution is amenable from a numerical point of view.
  In the case of anisotropic fluids, a structure-preserving finite element scheme was proven to converge to a dissipative solution, but the convergence to a measure-valued solution seems to be out of reach \cite[Rem.~3.7]{nematicelectro}. 
  Especially the high-order regularizations used to prove convergence to a measure-valued solution~\cite{masswertig} are not amenable from a numerical standpoint. 
   Moreover, in contrast to the set of weak solutions, the set of dissipative solutions is convex and weakly-$\ast$ closed in case of convex energy and dissipation potential. 
   This can be used to define selection criteria choosing a certain dissipative solution in order to achieve uniqueness~\cite{maxdiss}, 
   see also Remark~\ref{rem:convex} below.

The article at hand extends 
previous concepts
of dissipative solutions in two directions. 
Firstly, we do not fix the regularity weight $\mathcal K$ in the 
relative energy estimate \eqref{relen.intro}
defining energy-variational solutions.
This allows to consider different classes of test functions
and lets us derive different results for different choices of the regularity weight $\mathcal{K}$. 
For one class, we preserve the weak formulation of the Navier--Stokes-like part in the generalized formulation~(see Prop.~\ref{prop:envarweak}), and for another choice, we can deduce the convexity of the solution set (see Rem.~\ref{rem:convex}). 
As mentioned above, this convexity property 
may be used 
to select a unique physically relevant solution~\cite{envar,maxdiss}. 
Secondly,
we refrain from dropping the term $\mathcal W^{(\mathcal K)}$ in \eqref{relen.intro},
which for suitable $\mathcal K$ is non-negative and usually estimated by $0$.
By choosing $\mathcal K$ such that the relevant terms in $\mathcal W^{(\mathcal K)}$
are convex and lower semicontinuous, 
they can be kept when passing to the limit $\gamma\to0$,
thus making the solution concept more selective (see Remark~\ref{rem:formulation}).

The structure of this article is as follows. 
We first introduce the general notation
and prepare some auxiliary results in Section \ref{sec:preliminaries}.
In Section \ref{sec:gamma>0} we recall the existence theorem for generalized solutions to \eqref{eqvis}
in the case $\gamma>0$ and show the local-in-time existence of strong solutions.
Subsequently, we prove that strong solutions are unique in the class of generalized solutions.
For this purpose, we derive a suitable relative energy inequality.
In Section \ref{sec:ev.sol} we introduce the notion of energy-variational solutions to \eqref{eqvis}
in the case $\gamma=0$,
which is based on a similar relative energy inequality,
and we show their existence by an approximation with generalized solutions for $\gamma>0$.
Subsequently, we derive general properties of energy-variational solutions.
In the appendix we explain how to infer the present notion of generalized solutions
from the slightly different notion introduced in \cite{EiterHopfMielke2021}.

\section{Preliminaries}
\label{sec:preliminaries}
In this section we prepare the notation used throughout this article 
as well as some helpful inequalities.
We further introduce the basic regularity hypotheses on the data assumed throughout this manuscript.

\subsection{Notations}
\label{ssec:notations}

\paragraph{General notations.} 

If $\f a=(\f a_j),\,\f b=(\f b_j)\in\R^3$ 
are two vectors, their inner product and their tensor product 
are denoted by $\f a\cdot \f b=\f a_j \f b_j$ 
and $\f a\otimes \f b$ with $\np{\f a\otimes \f b}_{jk}=\f a_j\f b_k$, respectively.
Here and in what follows,
we tacitly use Einstein's summation convention.
For the inner product of two second-order tensors 
$\mathbb{A}=(\mathbb{A}_{jk}),\,\mathbb{B}=(\mathbb{B}_{jk})\in\R^{3\times3}$,
we write $\mathbb{A}:\mathbb{B}= \mathbb{A}_{jk}\mathbb{B}_{jk}$,
and the inner product of two third-order tensors 
$\mathbb{C}=(\mathbb{C}_{jk\ell}),\,\mathbb{D}=(\mathbb{D}_{jk\ell})\in\R^{3\times3\times3}$
is denoted by $\mathbb{C}\dreidots\mathbb{D}= \mathbb{C}_{jk\ell}\mathbb{D}_{jk\ell}$.
The third-order tensor $\mathbb A\otimes \f a$ is
defined by $(\mathbb{A}\otimes \f a)_{jk\ell}=\mathbb{A}_{jk}\f a_\ell$.
By $\mathbb{A}^\transpose$ and $\trace\mathbb{A}$ we denote the transpose and the trace of $\mathbb{A}$,
and $\R_\delta^{3\times3}\coloneqq\setc{\mathbb{A}\in\R^{3\times3}}{\mathbb{A}^\transpose=\mathbb{A},
\trace\mathbb{A}=0}$ denotes the class of symmetric deviatoric matrices.
Moreover, $\zerovec\in\R^3$ and $\zerotens\in\R^{3\times3}$ 
denote the zero vector and the zero tensor, respectively.

The symbol $\Omega$ always denotes a bounded Lipschitz domain in $\R^3$,
and points in $(x,t)\in\OmT$, $T>0$, consist of a spatial variable $x\in\Omega$ and a
time variable $t\in(0,T)$.
By $\pt u$ and $\partial_j u\coloneqq\partial_{x_j}u$, $j=1,2,3$,
we denote time and spatial partial derivatives of a (sufficiently regular) function $u$.
We further write $\grad$ and $\Delta$ for the gradient and the Laplace operator.
For a vector field $\vvel=(\vvel_1,\vvel_2,\vvel_3)$, 
the symmetric and skew-symmetric parts of $\grad\vvel$
are denoted by
\[
\symmgrad{\vvel}\coloneqq \frac{1}{2}\np{\grad\vvel+\grad\vvel^\transpose}, 
\qquad
\skewgrad{\vvel}\coloneqq\frac{1}{2}\np{\grad\vvel-\grad\vvel^\transpose}.
\]
Moreover, $\Div\vvel=\partial_j\vvel_j$ denotes the divergence of $\vvel$,
and we set
$\vvel\cdot\grad u\coloneqq \np{\vvel\cdot\grad} u\coloneqq \vvel_j\partial_j u$.
The divergence $\Div\Stens$ of a tensor field $\Stens=(\Stens_{jk})$
is defined by 
$(\Div\Stens)_j=\partial_k \Stens_{jk}$.

We let $\norm{\cdot}_{X}$ denote the norm of a Banach space $X$.
When the dimension is clear from the context, we do not distinguish between $X$ and its 
$n$-fold Cartesian product $X^n$.
We further write $X^\ast$ for the dual space of $X$,
and $\langle \varphi,x\rangle$ 
denotes the duality pairing 
of $\varphi\in X^\ast$ and $x\in X$. 
When $X$ is a Hilbert space, 
we sometimes write $(x,y)$ for the inner product of two elements $x,y\in X$.

\paragraph{Function spaces.}

By $\CRi(\Omega)$ we denote the class of smooth functions in $\Omega$,
and $\CRci(\Omega)$ consists of all elements in $\CRi(\Omega)$ with compact support in $\Omega$.
Lebesgue and Sobolev spaces are denoted by $\LR{q}(\Omega)$ and $\WSR{k}{q}(\Omega)$
for $q\in[1,\infty]$ and $k\in\N$,
and we set $\HSR{k}(\Omega)\coloneqq\WSR{k}{2}(\Omega)$.
Moreover, $\HSRN{1}(\Omega)$ consists of all functions in $\HSR{1}(\Omega)$
with vanishing boundary trace,
and $\HSR{-1}(\Omega)\coloneqq(\HSRN{1}(\Omega))^\ast$ is the associated dual space
with respect to the distributional duality pairing.

For an interval $I\subset\R$,
the class of continuous $X$-valued functions is denoted by $\CR{0}(I;X)$.  
For the associated Bochner--Lebesgue spaces we write $\LR{q}(I;X)$,
and we define $\WSR{1}{q}(I;X)\coloneqq\setc{u\in\LR{q}(I;X)}{\pt u\in\LR{q}(I;X)}$.
As above, we set $\HSR{1}(I;X)\coloneqq\WSR{1}{2}(I;X)$,
and the classes $\LRloc{q}(I;X)$ and $\HSRloc{1}(I;X)$ 
contain all functions that,
when restricted to any compact subinterval $J\subset I$,
belong to $\LR{q}(J;X)$ or $\HSR{1}(J;X)$, respectively. 
If $I=(0,T)$, we simply write 
$\CR{0}\np{0,T;X}=\CR{0}\np{I;X}$ and $\LR{q}\np{0,T;X}=\LR{q}\np{I;X}$. 
Moreover, for functions $u$ on $\Omega \times I$,
we sometimes abbreviate $u(t) \coloneqq u(\,\cdot\,,t)$ for $t\in I$.

We let
$\CRcisigma\np{\Omega}
\coloneqq\setcl{\varphi\in\CRci(\Omega)^3}{\Div\varphi=0}$
denote the class of smooth solenoidal vector fields,
and the associated Lebesgue and Sobolev spaces are given by
\[
\begin{aligned}
\LRsigma{2}(\Omega)
&\coloneqq\setcl{\vvel\in\LR{2}(\Omega)^3}{\Div\vvel=0, \ 
\restriction{\vvel}{\partial\Omega}\cdot\nvec=0}
=\closure{\CRcisigma\np{\Omega}}{\norm{\cdot}_{L^2}},
\\
\HSRNsigma{1}(\Omega)
&\coloneqq\setcl{\vvel\in\HSRN{1}(\Omega)^3}{\Div\vvel=0}
=\closure{\CRcisigma\np{\Omega}}{\norm{\cdot}_{\HSR{1}}},
\end{aligned}
\]
where the conditions $\Div\vvel=0$ and $\restriction{\vvel}{\partial\Omega}\cdot\nvec=0$
have to be understood in a weak sense;
see \cite[Theorem III.2.3]{Galdi2011_SteadyStateNSE} for example.
If $I\subset\R$ is an interval, we further set
$\CRcisigma\np{\Omega \times I}
\coloneqq\setcl{\Phi\in\CRci(\Omega\times I)^3}{\Div\Phi=0}$.
Moreover, we introduce the spaces of 
symmetric deviatoric fields
\[
\begin{aligned}
\LRdev{2}(\Omega)
&\coloneqq\setcl{\Stens\in\LR{2}(\Omega)^{3\times3}}
{\Stens=\Stens^\transpose,\ \trace\Stens=0},
\\
\HSRdev{1}(\Omega)
&\coloneqq\setcl{\Stens\in\HSR{1}(\Omega)^{3\times3}}
{\Stens=\Stens^\transpose,\ \trace\Stens=0}.
\end{aligned}
\]
In view of the (formal) energy dissipation law~\eqref{eq:enbalance}, 
for $\gamma>0$ we seek solutions~$\np{\vvel,\Stens}$ of~\eqref{eqvis} in the natural energy space 
$\LHT\times\XT$ where
\[
\begin{aligned}
\LHT
&\coloneqq
\LRloc{\infty}\np{[0,T);\LRsigma{2}(\Omega)}
\cap\LRloc{2}\np{[0,T);\HSRN{1}(\Omega)^3},
\\
\XT
&\coloneqq
\LRloc{\infty}\np{[0,T);\LRdev{2}(\Omega)}
\cap\LRloc{2}\np{[0,T);\HSRdev{1}(\Omega)^{3\times3}},
\end{aligned}
\] 
Usually, the time $T\in(0,\infty]$ will be fixed and we simply write
\begin{equation}\label{eq:weaksolspace}
\Xb\coloneqq\LHT\times\XT
\end{equation}
for the solution space.
We further need function spaces obeying additionally a Serrin-type regularity criterion,
which are defined by
\[
\begin{aligned}
	\YT^{s}
	&\coloneqq
	\HSRloc{1}([0,T);(\HSRNsigma{1}(\Omega))^\ast)
	\cap \LRloc{2}([0,T);\HSRNsigma{1}(\Omega))\cap 
	\LRloc{s}([0,T);\LR{r}(\Om)^{3}),
	\\
	\ZT^{q}
	&\coloneqq
	\HSRloc{1}([0,T);(\HSR{1}_\delta(\Omega))^\ast)
	\cap \LRloc{2}([0,T);\HSRdev{1}(\Omega))\cap 
	\LRloc{q}([0,T);\LR{p}(\Om)^{3\times3})
\end{aligned}
\] 
for $p,q,r,s\in(1,\infty)$ satisfying
\begin{equation}
	\frac{2}{s}+\frac{3}{r} =  1 \quad \text{ as well as }\quad \frac{2}{q}+\frac{3}{p} =  1 \label{serrin}\,,
\end{equation}
and we introduce the space
\begin{equation}\label{eq:testfctspace}
\Yb\coloneqq\bigcup_{s,q\in(2,\infty)}\YT^{s}\times\ZT^{q}.
\end{equation}

\paragraph{Functionals.}
Let $X$ be a Banach space and $\potential\colon X\to[0,\infty]$ be convex and proper (that is, $\potential\not\equiv\infty$) .
Then $\partial\potential$ denotes the subdifferential of $\potential$
defined by
\[
\partial\potential(u)\coloneqq
\setcl{\varphi\in X^\ast}{\potential(v)\geq\potential(u)+\langle\varphi,v-u\rangle\text{ for all }v\in X}
\]
for $u\in X$. 
When $X$ is a Hilbert space, then we can identify $X^\ast$ 
with $X$ such that $\partial\potential(u)\subset X$.
Moreover, the convex conjugate of $\potential$ is denoted by $\potential^\ast\colon X^\ast\to[0,\infty]$.

On the state space $\statespace:= \LRsigma{2}(\Om)\times \LRdev{2}(\Om)$
we define the energy functional 
$\energy\colon\statespace\to [0,\infty)$ by 
\begin{equation}\label{eq:energy.def}
	\energy(\vvel,\Se):=\frac{1}{2} \| \f  v\|_{L^2(\Omega)}^2
	+ \frac{1}{2} \| \Se \|_{L^2(\Omega)}^2,
\end{equation}
and we let $\mathcal{R} \colon\statespace\times \statespace\ra [0,\infty) $ denote the associated relative energy functional
\begin{equation}\label{eq:relen.def}
	\mathcal{R}(\f v , \Se | \vv, \tS) =
	\frac{1}{2} \| \f  v- \vv \|_{L^2(\Omega)}^2
	+ \frac{1}{2} \| \Se - \tS \|_{L^2(\Omega)}^2 \,.
\end{equation}
Furthermore, given a functional  
$\mathcal{K}\colon\statespace\to [0,\infty] $,
we let
\begin{equation}\label{eq:DK.def}
	D_\mathcal{K}:=\left\{(\vvel,\Se)\in \timeES : \mathcal{K}(\vvel,\Se)\in \LRloc{1}([0,T))\right\}.
\end{equation}

\subsection{Technical inequalities}

\begin{lemma}\label{lem:ineq}
	Let $a\in H^1(\Omega) $, $b\in  H^1(\Omega)$ and $c \in L^r(\Omega)$, 
	$r\in(3,\infty)$.
	Then for every $\delta>0$, 
	there exists a constant $C_\delta>0$ such that 
	\[
	 \int_\Omega |a | |\nabla b | |  c |\dx \leq \delta \bp{\|  a \| ^2_{\HSR{1}(\Omega)}  + \| \nabla b \| ^2 _{\LR{2}(\Omega)} }+ C _\delta \| c \|_{\LR{r}(\Omega)}^s \| a \| _{\LR{2}(\Omega)} ^2  \,
	\]
	for $s\in(2,\infty)$ defined by $2/s+3/r=1$.
	If $a\in\HSRN{1}(\Omega)$, we even have
	\[
	 \int_\Omega |a | |\nabla b | |  c |\dx \leq \delta \bp{\| \nabla a \| ^2_{\LR{2}(\Omega)}  + \| \nabla b \| ^2 _{\LR{2}(\Omega)}} + C _\delta \| c \|_{\LR{r}(\Omega)}^s \| a \| _{\LR{2}(\Omega)} ^2  .
	\]
\end{lemma}
\begin{proof}
	First consider the case $a\in\HSRN{1}(\Omega)$.
	H\"older's, Gagliardo--Nirenberg's, and Young's inequalities provide the estimate
	\begin{align*}
		\begin{split}
			 \int_\Omega  \snorm{a}\,\snorm{\nabla  b} \,\snorm{c}\dx\leq{}&  \|  a  \|_{L^p(\Omega)} \| \nabla b  \|_{L^2(\Omega)} \| c  \|_{L^{2p/(p-2)}(\Omega)} \\ 
			\leq{}& c_p \| a  \|_{L^2(\Omega)} ^{(1-\alpha)}
			\| \nabla  a  \|_{L^2(\Omega)}^{\alpha}
			\| \nabla  b  \|_{L^2(\Omega)}
			\|  c \|_{L^{2p/(p-2)}(\Omega)} 
			\\ \leq{}& \delta \left ( \| \nabla  a \| ^2_{L^2(\Omega)}  + \| \nabla  b \| ^2 _{L^2(\Omega)}\right ) + C _\delta \|  c \|_{L^{2p/(p-2)}(\Omega)} ^{2/(1-\alpha)}
			\|  a\|_{L^2(\Omega)}^2 \,,
		\end{split}
	\end{align*}
	where $p\in(2,\infty)$ and $\alpha\in(0,1)$ is chosen according to Gagliardo--Nirenberg's inequality by 
	$$\alpha = 3 (p-2)/2p\quad\text{for}\quad  3<2p/(p-2)\,.$$ 
	Letting $r=2p/(p-2)$ concludes the proof for $a\in\HSRN{1}(\Omega)$.
	If we only have $a\in\HSR{1}(\Omega)$,
	then we can merely apply the Gagliardo--Nirenberg inequality with 
	$\|  a  \|_{\HSR{1}(\Omega)}$
	instead of
	$\| \nabla  a  \|_{L^2(\Omega)}$,
	which yields the asserted inequality in this case.
\end{proof}

\begin{lemma}\label{lem:invar}
	Let $g_0\in\mathbb{R}$. Let $f\in L^1(0,T)$ and $g\in L^\infty(0,T)$ with $g\geq 0$ a.e.~in $(0,T)$.
	Then the following two inequalities are equivalent:
	\begin{equation}\label{ineq1}
		-\int_0^T \phi'(t) g(t) \de t  + \int_0^T \phi(t) f(t) \de t \leq g_0
		\qquad \forall\phi \in \tilde{\C} ([0,T]),
	\end{equation}
	where $\tilde{\C} ([0,T]):=\{\phi\in\C^1([0,T]):\phi \geq 0,\phi'\leq 0, \phi(0)=1, \phi(T)=0 \}$,
	and 
	\begin{equation}
		g(t) + \int_0^t f(s) \de s \leq g_0 \quad \text{for a.e.~}t\in(0,T)\,\label{ineq2}.
	\end{equation}
	This equivalence remains valid if we replace $\tilde\C([0,T])$ by
	$\tilde{W}((0,T))\coloneqq\setc{\phi\in\WSR{1}{1}((0,T))}{
	\phi \geq 0,\,\phi'\leq 0 \text{ a.e.}, \, \phi(0)=1, \,\phi(T)=0 } $.
\end{lemma}
\begin{proof}
See~\cite[Lemma~2.4]{maxdiss}.
\end{proof}

\subsection{General hypotheses}

Throughout this manuscript, 
we investigate \eqref{eqvis} under the following assumptions.

\begin{hypothesis}\label{hypo}
We let $\Omega\subset\R^3$ be a bounded Lipschitz domain
and  $T\in(0,\infty]$. 
Moreover, $\potential\colon\LRdev{2}(\Omega)\to[0,\infty]$ 
denotes a convex, lower semicontinuous functional that satisfies $\potential(\zerotens)=0$.
In particular, $\potential$ is weakly lower semicontinuous in $\LRdev{2}(\Omega)$.
For the remaining data we assume the regularity
\begin{equation}\label{el:data}
	\vvel_0\in\LRsigma{2}(\Omega), \qquad
	\Stens_0\in\LRdev{2}(\Omega), \qquad
	\vf\in\LRloc{2}\np{[0,T);\HSR{-1}(\Omega)^3}.
\end{equation}
\end{hypothesis}

\section{Generalized and strong solutions in the case of stress diffusion}
\label{sec:gamma>0}

In this section, we investigate system~\eqref{eqvis} for $\gamma>0$, 
that is, when stress diffusion is present.
We begin by introducing the notions of strong and generalized solutions, and show that 
for sufficiently regular data with $\mathcal{P}(\Se_0)<\infty$ 
a strong solution exists at least locally in time.
The global-in-time existence of generalized solutions has been established in~\cite{EiterHopfMielke2021}.
The second main result of this section concerns a weak-strong stability estimate for the 
relative energy between a generalized and a strong solution.
An important implication of this estimate is 
that any generalized solution coincides with the strong solution starting from the same initial data
as long as the latter exists (see Theorem~\ref{thm:weakstrong}).
The basis of the stability estimate is an inequality involving the relative energy between a generalized solution and an arbitrary sufficiently regular competitor taking the role of a test function, see Proposition~\ref{prop:relen}. 
We will take up this energy-variational inequality in Section~\ref{sec:ev.sol}, where it provides a framework for passing to the limit $\gamma\to0$.

\subsection{Definition of generalized and strong solutions}

For the definition of generalized solutions,
recall the definition of the energy space $\Xb$ from~\eqref{eq:weaksolspace}.

\begin{definition}[Generalized solution]\label{def:gensol} 
	We call a couple $(\vvel,\Stens)$ a generalized solution of
	system~\eqref{eqvis} if 
	$(\vvel,\Stens)\in \Xb=\LHT\times\XT$ and
	the following holds true:
	\begin{enumerate}
		\item 	The velocity field satisfies the weak formulation
		\begin{equation}\label{eq:weaksol.v}
			\begin{aligned}
				\int_0^T\int_\Omega \bb{
					-\vvel\cdot\pt\Phi
					+(\vvel\cdot\nabla)\vvel\cdot\Phi
					+\eta \Stens:\grad\Phi
					&+\mu\grad\vvel:\grad\Phi
				}\dx\dt
				\\
				&=\int_0^T\langle\vf,\Phi\rangle\dt
				+\int_\Omega\vvel_0\cdot\Phi(\cdot,0)\dx
			\end{aligned}
		\end{equation}
for all $\Phi\in\CRcisigma\np{\OmTzero}$, and the partial energy inequality
\begin{equation} 
	\frac{1}{2}\norm{\vvel(t)}_{L^2(\Omega)}^2
	+\mu\int_0^{t}\norm{\grad\vvel}_{L^2(\Omega)}^2\ds
	\leq\frac{1}{2} \norm{\vvel_0}_{L^2(\Omega)}^2
	+\int_0^{t}\! \langle\vf,\vvel\rangle\ds
	-\int_0^{t}\!\int_\Omega\eta\Stens:\grad\vvel \dx\ds
	\label{est:enineq.v}
\end{equation}
is satisfied for almost all $t\in(0,T)$.
\item The extra stress tensor $\Se$ satisfies the evolutionary variational inequality
\begin{equation}\label{est:varin0T}
\begin{aligned}
		\tfrac{1}{2}\|\Stens(t){-}\tStens(t)\|_{L^2(\Omega)}^2
		+\int_0^{t}\!\!\langle\pt\tStens;\Stens{-}\tStens\rangle
		{+}\potential(\Stens){-}\potential(\tStens)\ds
		&+\int_0^{t}\!\!\!\int_\Omega\!\gamma\grad\Stens:\grad(\Stens{-}\tStens)\dx\ds 
		\\
		-\int_0^{t}\!\int_\Omega \bp{\vvel\cdot\grad\Stens
			+\Stens\skewgrad\vvel-\skewgrad\vvel\Stens}:\tStens 
		&-\eta\symmgrad\vvel:(\Stens-\tStens)\dx\ds\\
		&\qquad\ 
		\leq\tfrac{1}{2}\|\Stens_0-\tStens(0)\|_{L^2(\Omega)}^2,
	\end{aligned}
\end{equation}
for all $\tStens \in\ZT^q$, $q\in(2,\infty)$, and a.a.~$t\in(0,T)$.
\end{enumerate}

\end{definition}

In the article~\cite{EiterHopfMielke2021}, existence for~\eqref{eqvis} was shown based on a slightly weaker form of the notion of generalized solutions than introduced above. The main difference lies in the fact that the first term in~\eqref{est:varin0T}, that is, the partial relative energy 
at time $t$, does not appear in the notion used in~\cite{EiterHopfMielke2021}. The above version including this term
is essentially equivalent (cf.~Lemma~\ref{l:0=>0T}) but better suited for the relative energy methods used in the present paper.

\begin{theorem}[Existence of generalized solutions~\cite{EiterHopfMielke2021}]\label{thm:gensol}
	Assume that Hypothesis~\ref{hypo} is satisfied.
	Then there exists a generalized solution $\np{\vvel,\Stens}\in\Xb$ 
	in the sense of Definition \ref{def:gensol}.
\end{theorem}

For the most part, Theorem~\ref{thm:gensol} was established in~\cite{EiterHopfMielke2021}. Details on how to infer the version above, which is formulated with the current, upgraded version of generalized solutions, are provided in Appendix~\ref{sec:appendix}. 

\begin{remark}\label{rem:enineq.gensol}
	Setting $\tStens=0$ in \eqref{est:varin0T}
	yields the partial energy inequality
	\begin{equation}\label{est:enineq.S}
		\frac{1}{2}\|\Stens(t)\|_{L^2(\Omega)}^2
		+\int_0^{t}\!\bp{\gamma\norm{\grad\Stens}_{L^2(\Omega)}^2 +\potential(\Stens)}\ds
		\\
		\leq\frac{1}{2}\|\Stens_0\|_{L^2(\Omega)}^2
		+\int_0^{t}\!\int_\Omega\eta\symmgrad\vvel:\Stens\dx\ds.
	\end{equation}
	Since we have $\symmgrad\vvel:\Stens=\grad\vvel:\Stens$ by the symmetry of $\Stens$,
	summation of \eqref{est:enineq.v} and \eqref{est:enineq.S}
	further yields the total energy-dissipation inequality
	\begin{equation} 
			\energy(\vvel(t),\Stens(t))
			+\int_0^{t}\! \bp{\mu\norm{\grad\vvel}_{L^2(\Omega)}^2
				+\gamma\norm{\grad\Stens}_{L^2(\Omega)}^2 
				+\potential(\Stens)}\ds
			\leq\energy(\vvel_0,\Stens_0)
			+\int_0^{t}\!\langle\vf,\vvel\rangle \ds,
	\label{est:enineq}
	\end{equation}
	where the energy functional $\energy$ was introduced in \eqref{eq:energy.def}.
\end{remark}

\begin{definition}[Strong solution]\label{def:strong.vS}
	We call $(\vvel,\Stens)\in\Xb=\LHT\times\XT$ 
	a strong solution of problem \eqref{eqvis} with initial data $(\vvel_0,\Se_0)$ if the following holds:
	\begin{itemize}
		\item $\vvel\in\LRloc{s}([0,T),\LR{r}(\Omega)^3)$ for some $r,s\in(1,\infty)$ 
		with $2/s+3/r=1$, and equation~\eqref{eqvis1} for the velocity component is satisfied in the weak sense, \textit{i.e.}, equation~\eqref{eq:weaksol.v} holds true
		for all $\Phi\in\CRcisigma\np{\OmTzero}$.
		\item $\Stens\in\ZT^q$ for some $q\in(1,\infty)$, $\Se(0)=\Se_0$, 
		and 
		for all $\mathbb{T}\in \HSRdev{1}(\Omega)$
		and a.a.\ $t\in (0,T)$ it holds
		\begin{align}\label{eq:100}
		\begin{multlined}
			\langle \t \Se(t),\Stens(t)- \mathbb{T}\rangle
			+\gamma\int_\Om\nabla\Stens(t)\dreidots\bp{\nabla(\Stens(t)-\mathbb{T})}\dx
			+\mathcal{P}(\Stens(t))-\mathcal{P}(\mathbb{T})
			\\
			+\int_\Om
			\big( (\f v \cdot \nabla) \Se + \Se (\nabla \f v )_{\skw} -(\nabla \f v)_{\skw} \Se	-\eta(\nabla \vvel)_{\sym}\big)\Big|_t:\big(\Stens(t)-\mathbb{T}\big)
		\dx
			\leq0.
		\end{multlined}		
\end{align}
	\end{itemize}
\end{definition}
Lemma~\ref{lem:ineq} ensures that the first and the second term in the first line and the term in the second line of inequality~\eqref{eq:100} are in $L^1_\loc([0,T))$ and thus, in particular, finite a.e.~in $(0,T)$. Choosing $\mathbb{T}=\mathbb{O}$ further shows that $\mathcal{P}(\Se)\in L^1_\loc([0,T))$ whenever $(\vvel,\Stens)$ is a strong solution,
so that, in particular, $\potential(\Stens(t))<\infty$ for a.a.~$t\in(0,T)$.

Moreover, the velocity field $\vvel$ of a strong solution $\np{\vvel,\Stens}$
in the sense of Definition \ref{def:strong.vS}
is a weak solution 
to \eqref{eqvis1}
that satisfies a Serrin condition.
This is in accordance with the well-known notion of strong solutions 
to the classical Navier--Stokes equations
as used in \cite{Sohr_NSE_2001} for example.
In particular,
due to this regularity condition,
it is not necessary to additionally assume that $\vvel$
satisfies the partial energy inequality
\eqref{est:enineq.v}
since the corresponding energy \emph{equality} is satisfied automatically,
that is, 
\begin{equation}
	\label{eq:energy.v}
	\frac{1}{2}\norm{\vvel(t)}_{L^2(\Omega)}^2
	+\mu\int_0^{t}\norm{\grad\vvel}_{L^2(\Omega)}^2\ds
	=\frac{1}{2} \norm{\vvel_0}_{L^2(\Omega)}^2
	+\int_0^{t}\! \langle\vf,\vvel\rangle\ds
	-\int_0^{t}\!\int_\Omega\eta\Stens:\symmgrad\vvel \dx\ds.
\end{equation}
As a consequence of these observations, we infer the following consistency property.
\begin{remark}[Strong solutions are generalized solutions]
Any strong solution $(\vvel,\Se)$ in the sense of Definition~\ref{def:strong.vS} is a generalized solution in the sense of Definition~\ref{def:gensol}. For the velocity component, this follows from the fact that strong solutions satisfy the energy equality~\eqref{eq:energy.v}. Inequality~\eqref{est:varin0T} for given $\tS\in Z^q_T$ is obtained upon integrating 
 inequality~\eqref{eq:100} at time $s$ and with the choice $\mathbb{T}:=\tS(s)$ in time from $s=0$ to $s=t$, where one uses the identity
	\begin{align*}
		\int_0^t\langle \t \Se,\Stens{-} \tS\rangle\ddd s 
		= \frac{1}{2}\|\Stens(t){-}\tStens(t)\|_{L^2(\Omega)}^2
		-	\frac{1}{2}\|\Stens_0{-}\tStens(0)\|_{L^2(\Omega)}^2
		+\int_0^t\langle \t\tS,\Stens{-} \tS\rangle\ddd s.
	\end{align*}
\end{remark}
Let us finally point out the relation between the variational inequality~\eqref{eq:100} and the differential inclusion~\eqref{eqvis2}. For this purpose let $\mathcal{\widehat P}:\HSRdev{1}(\Omega)\to[0,\infty]$ be the restriction of $\mathcal{P}$ to $\HSRdev{1}(\Omega)$.
By definition,
its convex subdifferential $\partial\mathcal{\widehat P}$ maps elements of $\HSRdev{1}(\Omega)$ to subsets of $\HSRdev{1}(\Omega)^*$, and for almost all $t\in(0,T)$, inequality~\eqref{eq:100}  can be written as
a differential inclusion
in $\HSRdev{1}(\Omega)^*$, namely
\begin{equation}
-\big(	\t \Se + (\f v \cdot \nabla) \Se + \Se (\nabla \f v )_{\skw} -(\nabla \f v)_{\skw} \Se -\gamma \Delta \Se-  \eta(\nabla \f v)_{\sym} \big) \in \partial \mathcal{\widehat P}(\Se)
\label{inclusion}
\end{equation}
with the understanding that
$\langle-\gamma \Delta \Se,\Ttens\rangle:=\gamma\int_\Om\nabla \Se \dreidots\nabla\Ttens\dx$.

 In Section~\ref{ssec:strong} we will construct local-in-time strong solutions enjoying the extra regularity $\Delta \Se (t)$, $\t \Se(t)  \in \LRdev{2}(\Omega)$, which satisfy the inclusion~\eqref{inclusion} in the $\LRdev{2}(\Omega)$ sense as well
 (\textit{cf.}~Remark~\ref{rem:addreg}). 

\begin{remark}[Energy equality]
We note that for a strong solution $( \f v , \Se)$, we may recover the  energy equality~\eqref{eq:enbalance}. Indeed, the inclusion~\eqref{inclusion} implies that there exists $\mathbb{G}$ with $\mathbb{G}(t)\in  \partial \mathcal{\widehat P}(\Se(t))$ for a.a.~$t\in(0,T)$ such that the equality
\begin{align*}
\t \Se + ( \f v \cdot \nabla) \Se + \Se (\nabla \f v )_{\skw} - (\nabla \f v)_{\skw} \Se - \gamma \Delta \Se + \mathbb{G} = \eta(\nabla \f v)_{\sym} \,
\end{align*}
holds in $(\HSR{1}_\delta(\Omega))^\ast$ a.e.~in $(0,T)$. 
Thanks to the extra regularity,
we may test this identity with~$\Se$.
 Integrating the resulting equality in time leads to the partial energy equality
\begin{align*}
\frac{1}{2}\| \Se (t) \|_{L^2(\Omega)}^2 + \!\int_0^t \!\gamma \| \nabla \Se \|_{L^2(\Omega)}^2 + \mathcal{{\widehat P}}(\Se)+ \mathcal{{\widehat P}}^* (\mathbb{G}) \de s = \frac{1}{2}\| \Se_0\|_{L^2(\Omega)}^2 +\! \int_0^t \!\int_\Omega \!\eta \Se : ( \nabla \f v)_{\sym} \dx\de s \,
\end{align*}
for a.e.~$t\in (0,T)$, where we used the Fenchel  identity $ \langle \mathbb{G}, \mathbb{S}\rangle =  \mathcal{{\widehat P}}(\Se)+ \mathcal{{\widehat P}}^* (\mathbb{G})$ for $\mathbb{G}\in \partial  \mathcal{{\widehat P}}(\Se)$. 
Furthermore, 
we may use $\vvel$ as a test function in \eqref{eq:weaksol.v},
which leads to the partial energy equality
\eqref{eq:energy.v}
for a.a.~$t\in(0,T)$. 
Summing up the two partial energy equalities, we arrive at the energy equality~\eqref{eq:enbalance}.
\end{remark}

\subsection{Local existence of strong solutions}\label{ssec:strong}
In the following, we show that under an additional regularity hypothesis on the data, problem~\eqref{eqvis} has a strong solution on a time interval $(0,T)$ provided $T>0$ is small enough.
For simplicity, we  only consider the case $\vf\equiv0$ here, but the result equally holds for forcings  $\vf\in\LRloc{2}\np{[0,\infty);\LR{2}(\Omega)^3}.$

\begin{theorem}[Local existence of strong solutions]
\label{thm:strongsol}
Let  $\gamma>0,\mu>0$ and $\eta\ge0$.  
Suppose that $\Om\subset\mathbb{R}^3$ is a bounded domain with $\CR{2}$ boundary, 
and let $\vf\equiv0$. In addition to the basic hypotheses~\eqref{el:data} assume that 
\begin{align}\label{eq:data.reg}
	\vvel_0\in H^1(\Om)^3,\;\;  \Stens_0\in H^1(\Om)^{3\times 3} \quad \text{ and  }\quad\mathcal{P}(\Stens_0)<\infty.
\end{align}
Then there exists $T\in(0,\infty)$
 such that problem~\eqref{eqvis} has at least one strong solution $(\vvel,\Stens)$ with initial data $(\vvel_0,\Se_0)$.
This solution enjoys the following additional regularity properties:
\begin{align}\label{eq:149}
	\begin{aligned}
    \vvel&\in H^1(0,T;L^2(\Omega)^3)\cap L^\infty(0,T;H^1(\Omega)^3)\cap L^2(0,T;H^2(\Omega)^3),
	\\\Stens&\in H^1(0,T;L^2(\Omega)^{3\times 3})\cap L^\infty(0,T;H^1(\Omega)^{3\times 3}),
	\quad \mathcal{P}(\Stens)\in L^\infty(0,T).
\end{aligned}
\end{align}
\end{theorem}

\begin{proof}
The main point in the proof is to derive sufficiently strong \textit{a priori} bounds (cf.~\eqref{eq:apri.s} below) for suitably regular solutions of~\eqref{eqvis} on some short time interval $[0,T]$ in the case that $\mathcal{P}$ is Fr\'echet differentiable with globally Lipschitz continuous derivative. 
For the actual construction of a strong solution on $(0,T)\times\Om$, 
one may then follow the roadmap in~\cite{EiterHopfMielke2021}: 
One first regularizes the potential $\mathcal{P}\colon\LRdev{2}(\Omega)\to[0,\infty]$ by its Moreau envelope and then performs a Galerkin approximation for the regularized potentials similar to~\cite[Section~4]{EiterHopfMielke2021}. Along the approximate sequence, the \textit{a priori} estimates
can be derived rigorously. Moreover, they allow us to pass to the limit along the approximate solutions (up to a subsequence), not only in the weak formulation for $(\vvel,\Stens)$ but also in the strong formulation~\eqref{eq:100} for the internal stress tensor, and imply the asserted regularity~\eqref{eq:149}.
Let us note that for the velocity component,
in the Galerkin approximation one should here use a Galerkin basis  composed of eigenfunctions of the Stokes operator $-P\Delta$ on $\LRsigma{2}(\Omega)$
with $P$ denoting the Helmholtz projector $P\colon\LR{2}(\Omega)^3\to\LRsigma{2}(\Omega)$ (rather than the smooth basis functions introduced in~\cite{EiterHopfMielke2021} for constructing weak solutions).
The eigenfunctions of the Stokes operator only belong to $\HSRNsigma{1}(\Omega)\cap\HSR{2}(\Omega)^3$ in general, but this regularity suffices for our purpose.

In the rest of the proof, we will show that, in addition to the fundamental energy estimate
\begin{multline}\label{eq:150}
	\begin{multlined}
	\sup_{t\in[0,T]}\bigg(\frac{1}{2}\|\vvel(t)\|_{{L^2(\Omega)}}^2+\frac{1}{2}\| \Stens(t)\|_{L^2(\Omega)}^2\bigg)
	\\\hspace{.15\linewidth}+\mu\|\nabla \vvel\|_{L^2(0,T;L^2(\Omega))}^2+\gamma\|\nabla \Stens\|_{L^2(0,T;L^2(\Omega))}^2+\int_0^T\mathcal{P}(\Stens(t))\,\ddd t \;\le\; \cten
\end{multlined}\hfill
\end{multline}
with $\cten:=\tfrac{1}{2}\|\vvel_0\|_{L^2(\Omega)}^2+\tfrac{1}{2}\| \Stens_0\|_{L^2(\Omega)}^2,$
regular solutions of~\eqref{eqvis} enjoy the following bound:
\begin{align}\label{eq:apri.s}
	\begin{aligned}
	\sup_{t\in[0,T]}\bigg(\|\nabla\vvel(t)\|_{L^2(\Omega)}^2
	&+\gamma\|\nabla\Stens(t)\|_{L^2(\Omega)}^2
	+\mathcal{P}(\Stens(t))\bigg)
	\\&+\mu\int_0^{T}\bigg(\|\nabla^2\vvel(t)\|_{L^2(\Omega)}^2
		+\|\partial_t\Stens(t)\|_{L^2(\Omega)}^2\bigg)\,\ddd t 
\;	\le \;\ctv,
\end{aligned}
\end{align}
for
$T=T(\|\vvel_0\|_{H^1(\Omega)},\| \Stens_0\|_{L^2(\Omega)}, \sqrt{\gamma}\| \nabla\Stens_0\|_{L^2(\Omega)},\potential(\Se_0),\gamma^{-1},\mu^{-1},\eta)>0$ small enough
and a finite constant
\begin{align}\label{eq:def.ct}
	\ctv&=C(\|\vvel_0\|_{H^1(\Omega)},\| \Stens_0\|_{L^2(\Omega)},
	\sqrt{\gamma}\| \nabla\Stens_0\|_{L^2(\Omega)},\mathcal{P}(\Stens_0),\gamma^{-1},\mu^{-1},\eta).
\end{align}
Here, the quantity $\ctv$ can be chosen to be non-decreasing in each of its arguments, while the time $T(\cdot)>0$ is non-increasing in each of its arguments.

In virtue of the embedding $\LR{2}(0,T;\HSR{2}(\Omega)^3)\hookrightarrow\LR{2}(0,T;\LR{\infty}(\Omega)^3)$,
estimate \eqref{eq:apri.s} further yields $\vvel\cdot\grad\vvel\in\LR{2}(0,T;\LR{2}(\Omega)^3)$,
which implies that $\pt\vvel\in\LR{2}(0,T;\LR{2}(\Omega)^3)$ (for details see e.g.~\cite[Lemma~6.2]{RRS_2016}).
Hence, the regularity asserted in \eqref{eq:149} indeed follows from \eqref{eq:apri.s}.

For the derivation of~\eqref{eq:apri.s}, 
we first test the Navier--Stokes equations~\eqref{eqvis1} with $-P\Delta\vvel(t)\in\LRsigma{2}(\Omega)$ to infer
\[
\frac{1}{2}\frac{\ddd}{\ddd t}\|\nabla\vvel\|_{L^2(\Omega)}^2+\mu\|P\Delta\vvel\|_{L^2(\Omega)}^2
=\int_{\Om}( \f v \cdot \nabla ) \f v \cdot P\Delta\vvel\;\ddd x
-\eta\int_{\Om}(\nabla\cdot\Stens)\cdot P\Delta\vvel\;\ddd x
=: I_1+I_2.
\]
Using Agmon's inequality as well as the Poincar\'e and Young inequalities
and the elliptic estimate 
$\norm{\vvel}_{H^2(\Omega)}\leq C\norm{P\Delta\vvel}_{\LR{2}(\Omega)}$ (see e.g.~\cite[Proposition~4.7]{CF_1988}), 
the integral terms can be estimated as follows:
\[
\begin{aligned}
	|I_1|&\le \|\vvel\|_{L^\infty(\Omega)} \|\nabla\vvel\|_{L^2(\Omega)} \|P\Delta\vvel\|_{L^2(\Omega)}
	\le  C\|\nabla\vvel\|_{L^2(\Omega)}^\frac{3}{2}
	\|P\Delta\vvel\|_{L^2(\Omega)}^\frac{3}{2}
	\\& \le  C\mu^{-3}\|\nabla\vvel\|_{L^2(\Omega)}^6+\frac{\mu}{4}\|P\Delta\vvel\|_{L^2(\Omega)}^2,
	\\
	|I_2|&\le C\mu^{-1}\eta^2\|\nabla \Stens\|_{L^2(\Omega)}^2+\frac{\mu}{4}\|P\Delta\vvel\|_{L^2(\Omega)}^2.
\end{aligned}
\]
Hence,
\begin{equation}\label{eq:v'}
	\frac{\ddd}{\ddd t}\|\nabla\vvel\|_{L^2(\Omega)}^2+\mu\|P\Delta\vvel\|_{L^2(\Omega)}^2
\le C\mu^{-3}\|\nabla\vvel\|_{L^2(\Omega)}^6+
C\mu^{-1}\eta^2\|\nabla \Stens\|_{L^2(\Omega)}^2.
\end{equation}
To proceed, let us recall that we may assume without loss of generality that $\partial\mathcal{P}(\Stens)$ is globally Lipschitz continuous. 
As mentioned above, the rigorous version of this step is to be carried out with the Moreau envelope of the nonsmooth dissipation potential, cf.~\cite{EiterHopfMielke2021}.
We may then test the evolution law \eqref{eqvis2} for the tensorial component with $\partial_t\Stens$ to find
\begin{multline*}
	\frac{\gamma}{2}\frac{\ddd}{\ddd t}\|\nabla\Stens\|_{L^2}^2+\|\partial_t\Stens\|_{L^2}^2
	+\int_\Om\partial\mathcal{P}(\Stens):\partial_t\Stens\,\ddd x
	\\\qquad=\int_\Om(\f v \cdot \nabla) \Se :\partial_t\Stens\,\ddd x
	+ \int_\Om\big(\Se (\nabla \f v )_{\skw} -(\nabla \f v)_{\skw} \Se\big):\partial_t\Stens\,\ddd x + \int_\Om \eta(\nabla \f v)_{\sym} :\partial_t\Stens\,\ddd x
	\\=: I_3+I_4+I_5.
\end{multline*}
The integral $I_3$ involving the convective term is estimated, using Agmon's inequality
and the bound 
$\norm{\vvel}_{H^2(\Omega)}\leq C\norm{P\Delta\vvel}_{\LR{2}(\Omega)}$, as 
\begin{align*}
	|I_3|\le \|\vvel\|_{L^\infty(\Omega)}	\|\nabla\Stens\|_{L^2(\Omega)}\|\partial_t\Stens\|_{L^2(\Omega)}
	&\le  C\|\nabla\vvel\|_{L^2(\Omega)}^\frac{1}{2}
	\norm{P\Delta\vvel}_{\LR{2}(\Omega)}^\frac{1}{2}
	\|\nabla\Stens\|_{L^2(\Omega)}\|\partial_t\Stens\|_{L^2(\Omega)}
	\\&\le   		C\|\nabla\vvel\|_{L^2(\Omega)}^2
	\|\nabla\Stens\|_{L^2(\Omega)}^4 + \frac{\mu}{4} 	\norm{P\Delta\vvel}_{\LR{2}(\Omega)}^2
	+ \frac{1}{6}\|\partial_t\Stens\|_{L^2(\Omega)}^2.
\end{align*}
For $I_4$ we estimate, using the Gagliardo--Nirenberg--Sobolev inequality,
\[
\begin{aligned}
	|I_4| \le  2\|\nabla\vvel\|_{L^3(\Omega)}\|\Stens\|_{L^6(\Omega)}\|\partial_t\Stens\|_{L^2(\Omega)}
	&\le C\|\nabla\vvel\|_{L^2(\Omega)}^\frac{1}{2}
	\norm{P\Delta\vvel}_{\LR{2}(\Omega)}^\frac{1}{2}
	\|\Stens\|_{H^1(\Omega)}\|\partial_t\Stens\|_{L^2(\Omega)}
	\\&\le C\|\nabla\vvel\|_{L^2(\Omega)}^2
	\|\Stens\|_{H^1(\Omega)}^4+ \frac{\mu}{4} 	\norm{P\Delta\vvel}_{\LR{2}(\Omega)}^2	+ \frac{1}{6}\|\partial_t\Stens\|_{L^2(\Omega)}^2.
\end{aligned}	
\]
The estimate for $I_5$ is immediate
\[
|I_5|\le C\eta^2\|\nabla\vvel\|_{L^2(\Omega)}^2+ \frac{1}{6}\|\partial_t\Stens\|_{L^2(\Omega)}^2.
\]
In combination, we deduce
\begin{equation}
\begin{aligned}\label{eq:S'}
	&\frac{\gamma}{2}\frac{\ddd}{\ddd t}\|\nabla\Stens\|_{L^2(\Omega)}^2
	+\frac{1}{2}\|\partial_t\Stens\|_{L^2(\Omega)}^2
+\frac{\ddd}{\ddd t}\mathcal{P}(\Stens)
\\
&\qquad\le 
C\|\nabla\vvel\|_{L^2(\Omega)}^2
\|\Stens\|_{H^1(\Omega)}^4
+ \frac{\mu}{2} 	\norm{P\Delta\vvel}_{\LR{2}(\Omega)}^2
+C\eta^2\|\nabla\vvel\|_{L^2(\Omega)}^2,
\end{aligned}
\end{equation}
where we used the regularity of $\mathcal{P}$ and $\Stens$ and the chain rule to rewrite the term involving the dissipation potential.

Adding up inequalities~\eqref{eq:v'} and~\eqref{eq:S'} gives
\begin{align*}
	&\frac{\ddd}{\ddd t}\bigg(\|\nabla\vvel\|_{L^2(\Omega)}^2
	+\frac{\gamma}{2}\|\nabla\Stens\|_{L^2(\Omega)}^2
	+\mathcal{P}(\Stens)\bigg)	
	+\frac{\mu}{2}\|P\Delta\vvel\|_{L^2(\Omega)}^2
	+\frac{1}{2}\|\partial_t\Stens\|_{L^2(\Omega)}^2
	\\&	\le C\mu^{-3}\|\nabla\vvel\|_{L^2(\Omega)}^6+
	C\mu^{-1}\eta^2\|\nabla \Stens\|_{L^2(\Omega)}^2
	+C\mu^{-1}\|\nabla\vvel\|_{L^2(\Omega)}^2
	\|\Stens\|_{H^1(\Omega)}^4
	+C\eta^2\|\nabla\vvel\|_{L^2(\Omega)}^2.
\end{align*}
Expanding
$\|\Stens\|_{H^1(\Omega)}^2=\|\nabla\Stens\|_{L^2(\Omega)}^2+\|\Stens\|_{L^2(\Omega)}^2$, we find that the function
$$\psi(t):=\|\nabla\vvel(t)\|_{L^2(\Omega)}^2
+\frac{\gamma}{2}\|\nabla\Stens(t)\|_{L^2(\Omega)}^2
+\mathcal{P}(\Stens(t))+1$$ satisfies the differential inequality
\[
	\frac{\ddd}{\ddd t}\psi(t)\le A\psi(t)^3,
\]
where
\[
	A:=	C\mu^{-3}+C\mu^{-1}\eta^2\gamma^{-1}
	+C\mu^{-1}\gamma^{-2}+C\mu^{-1}\cten^{2}+C\eta^2.
\]
Since, by hypothesis, $\psi(0)$ is finite, we may use a classical ODE comparison argument to deduce the
 existence of a time $$T=T(\|\vvel_0\|_{H^1(\Omega)},
 \| \Stens_0\|_{L^2(\Omega)},
 \sqrt{\gamma}\|\nabla\Stens_0\|_{L^2(\Omega)},\mathcal{P}(\Stens_0)
 ,\gamma^{-1},\mu^{-1},\eta)>0$$
 and a finite constant $\ctv$ as in~\eqref{eq:def.ct}
 such that $\psi(t)\le \ctv$ for all $t\in[0,T]$. 
Combined with the above estimates for $\mu\|P\Delta\vvel\|_{L^2(\Omega)}^2$ and $\|\partial_t\Stens\|_{L^2(\Omega)}^2$, this implies the bound~\eqref{eq:apri.s}.
\end{proof}

When $\potential\equiv0$, the evolution law for the tensorial component implies that any strong solution $(\vvel,\Se)$ with the regularity~\eqref{eq:149} further satisfies $\Se \in L^2(0,T; H^2(\Omega)^{3\times3})$. The following remark shows that this conclusion continues to hold true for nonsmooth potentials $\potential$ given by an integral functional.

\begin{remark}[$H^2$ regularity for $\Se$]\label{rem:addreg}

Suppose that the dissipation potential $\potential$ takes the form of an integral  $\potential(\Stens)=\int_\Omega\mathfrak{P}(\Stens(x))\dx$ for a
 lower semicontinuous, convex function $\mathfrak{P}:\R^{3\times 3}_{\delta}\to[0,\infty]$ with $\mathfrak{P}(\zerotens)=0$.
Then, under the hypotheses of Theorem~\ref{thm:strongsol}, strong solutions $(\vvel,\Se)$ satisfying~\eqref{eq:149} enjoy 
the additional regularity
\begin{align}\label{eq:H2reg}
	\Se \in L^2(0,T; H^2(\Omega)^{3\times 3}) 
\end{align}
and fulfill the differential inclusion~\eqref{eqvis2} in the $\LRdev{2}(\Omega)$ sense:
\[
\mathbb{G}\coloneqq - \bp{\t \Se + (\f v \cdot \nabla) \Se + \Se (\nabla \f v )_{\skw} -(\nabla \f v)_{\skw} \Se-\gamma\Delta\Se-\eta(\nabla \f v)_{\sym} }
\in L^2(0,T;\LRdev{2}(\Omega))
\]
with $\mathbb{G}(x,t)\in\partial\mathfrak{P}(\Se(x,t))$ for a.a.~$(x,t)\in\OmT$.

Let us first provide the formal argument demonstrating this assertion.
The estimate
\[
\begin{aligned}
&\left \| \eta(\nabla \f v)_{\sym} 
- (\f v \cdot\nabla) \Se - (\Se (\nabla \f v)_{\skw}- (\nabla \f v)_{\skw} \Stens) \right 
\|_{L^2(\Omega\times (0,T))}
\\
&\quad \leq \eta \| \nabla \f v \|_{L^2(\Omega\times (0,T) )} + \| \f v \| _{L^2(0,T;L^\infty(\Omega))} \| \Se \|_{L^\infty(0,T;H^1(\Omega))} + 2c \| \Se \|_{L^\infty(0,T;L^6(\Omega))} \| \f v \|_{L^2(0,T;H^2(\Omega))} \,
\end{aligned}
\]
combined with~\eqref{eq:150} and~\eqref{eq:apri.s} shows that the term 
$$\mathbb{F}:=\eta(\nabla \f v)_{\sym} -\t \Se 
- (\f v \cdot\nabla) \Se - (\Se (\nabla \f v)_{\skw}- (\nabla \f v)_{\skw} \Stens)$$
is controlled in $L^2(0,T;L^2(\Omega)^{3\times 3})$, whence
\begin{align}
 - \gamma \Delta \Se + \mathbb{G} = \mathbb{F} \in  L^2(0,T;\LRdev{2}(\Omega)).\label{eqrem}
\end{align}
Assuming for the moment that $\mathfrak{P}\in C^2$ and  testing~\eqref{eqrem} with $\D\mathfrak{P}(\Se)$ ($=\mathbb{G}$) yields
\begin{multline*}
 \int_0^T \int_\Omega  \gamma\nabla \Se\dreidots  \D^2\mathfrak{P}(\Se)\dreidots \nabla \Se  + |\D \mathfrak{P} (\Se) | ^2 \de x \de t 
	= \int_0^T\int_\Omega \mathbb{F}:  \D\mathfrak{P} (\Se) \de x \de t \\\leq \frac{1}{2}\|\D \mathfrak{P}(\Se)\|_{L^2(0,T;L^2(\Omega))} ^2 + \frac{1}{2}\| \mathbb{F}\|_{L^2(0,T;L^2(\Omega))}^2 \,,
\end{multline*}
where the first term on the left-hand side is non-negative thanks to the positive semidefiniteness of the Hesse form of the convex function $\mathfrak{P}$. 
We thus deduce the \textit{a priori} bound
\begin{align}\label{eq:H2.S}
	\|\D \mathfrak{P}(\Se)\|_{L^2(0,T;L^2(\Omega))} \le  \| \mathbb{F}\|_{L^2(0,T;L^2(\Omega))},
\end{align}
which further yields
$	\|\gamma\Delta\Se\|_{L^2(0,T;L^2(\Omega))} \le  C\| \mathbb{F}\|_{L^2(0,T;L^2(\Omega))}.$

To make the above reasoning rigorous, we may argue as follows. Following the proof of Theorem~\ref{thm:strongsol}, we first construct  local-in-time strong solutions using the regularized functions $\mathfrak{P}_{\ve,\lambda}:=\rho_\lambda\ast\mathfrak{P}_\ve$ with $\lambda,\ve\in(0,1]$.
Here, $\mathfrak{P}_\ve$ denotes the Moreau envelope of the nonsmooth function 
$\mathfrak{P}$, that is, 
\[
	\mathfrak{P}_\ve(\mathbb{T}) := \inf_{\mathbb{T'}\in\R^{3\times3}_\delta}
	\bigg(\mathfrak{P}(\mathbb{T'})+\frac{1}{2\ve}|\mathbb{T}-\mathbb{T'}|^2\bigg),
\]
which by construction satisfies the pointwise bound $\mathfrak{P}_\ve(\mathbb{T})\le \frac{1}{2\ve}|\mathbb{T}|^2$, while $\rho_\lambda$ denotes a standard (non-negative) mollifier. The functions $\mathfrak{P}_{\ve,\lambda}$ are $C^2$ and convex and satisfy the $\lambda$-uniform bound
\begin{equation}\label{eq:ptw.bd}
	0\le\mathfrak{P}_{\ve,\lambda}(\mathbb{T})\le \frac{C}{\ve}\left(|\mathbb{T}|^2+1\right)\text{ for all }\mathbb{T}\in \R^{3\times 3}_{\delta},
\end{equation}
whenever $\lambda\in(0,1].$
 (The fact that possibly $\mathfrak{P}_{\ve,\lambda}(0)\not=0$ does not affect the construction of solutions.) At the level of the strong solutions associated with $\mathfrak{P}_{\ve,\lambda}$, estimate~\eqref{eq:H2.S} and the $L^2$ bound for $\gamma\Delta\Se$ can be derived rigorously. The control~\eqref{eq:ptw.bd}  and the \textit{a priori} bounds allow to pass to the limit $\lambda\to0$ in the strong formulation~\eqref{eq:100} for all $\mathbb{T}\in \LRdev{2}(\Om)$, where the term $\gamma\int_\Om\nabla\Stens(t)\dreidots\bp{\nabla(\Stens(t)-\mathbb{T})}\dx$ is to be replaced by $-\gamma\int_\Om\Delta\Stens(t):\bp{\Stens(t)-\mathbb{T}}\dx$.
 Finally, in the limit $\ve\to0$, we obtain a strong solution on $(0,T)$ with the additional regularity~\eqref{eq:H2reg} satisfying the differential inclusion~\eqref{eqvis2} in the $\LRdev{2}$ sense.
 Since our weak-strong uniqueness principle (Theorem~\ref{thm:weakstrong} below) implies that strong solutions are unique, this shows that already the strong solutions obtained in Theorem~\ref{thm:strongsol} must have these regularity properties if the dissipation potential is in integral form.
\end{remark}

\subsection{Relative energy inequality and weak-strong stability}

The aim of this subsection is to establish the following weak-strong stability result,
which compares generalized solutions $(\vvel,\Stens)$
with strong solutions $\np{\tvvel,\tStens}$ to \eqref{eqvis}
in terms of the relative energy $\mathcal{R} (\vvel, \Stens| \tvvel,\tStens)$
defined in \eqref{eq:relen.def}.
It implies the uniqueness of strong solutions in the class of generalized solutions.

\begin{theorem}[Weak-strong stability]\label{thm:weakstrong}
Let $(\vv,\tS)$ be a strong solution according to Def.~\ref{def:strong.vS}
with initial data $(\vv_0,\tS_0)$.
Choose $p,q,r,s\in(2,\infty)$ satisfying \eqref{serrin}
such that 
$\tvvel\in\LRloc{s}([0,T);\LR{r}(\Omega))$,
$\tStens\in\LRloc{q}([0,T);\LR{p}(\Omega))$.
Then every generalized solution $(\f v , \Se)\in \Xb$ emanating from initial data $(\f v_0 , \Se_0)$
obeys the estimate
\begin{equation}\label{est:stability}
\begin{aligned}
	\mathcal{R} (\f v(t) , \Se(t)| \vv(t),\tS(t))
	+\int_0^t\Bp{\frac{\mu}{2} \| \nabla \f v - \nabla \vv\|_{L^2(\Omega)}^2
	&+ \frac{\gamma}{2} \| \nabla \Se-\nabla \tS \|_{L^2(\Omega)}^2} 
	\e^{\int_s^t \mathcal{K}(\vv,\tS) \de \tau}\,\ds
	\\
	&\qquad\quad\leq \mathcal{R} (\f v_0 , \Se_0| \vv_0,\tS_0)  \e^{\int_0^t \mathcal{K}(\vv,\tS) \de s } \,
\end{aligned}
\end{equation}
for a.e.~$t\in(0,T)$, where 
\begin{equation}\label{eq:Kstability}
\mathcal{K}(\vv,\tS)= C \left (\| \vv \|_{L^{r}(\Omega)}^s+\| \tS\|_{L^{p}(\Omega)}^q+\| \tS\|_{L^{p}(\Omega)}^2 \right )
\end{equation}
for a constant $C=C(\Omega,p,q,r,s,\mu,\gamma)>0$.
In particular, if the initial data $(\f v_0 , \Se_0)$ and $(\vv_0,\tS_0)$ coincide, then
$\vv \equiv \f v$ and $\tS=\Se$.
\end{theorem}

In the above formulation, we call the function~$\mathcal{K}$ the regularity weight since it measures the minimal regularity  of the function $(\vv,\tS)$ such that both sides of the relative energy inequality~\eqref{est:stability} remain finite. 

We will deduce Theorem \ref{thm:weakstrong}
from a suitable relative energy inequality
that compares a generalized solution with any sufficiently smooth function.
For its formulation, let us (formally) introduce
the operator $\mathcal{A}
=\big(\mathcal{A}^{(1)},\mathcal{A}^{(2)}\big)$ by
\begin{subequations}
\begin{align}
\label{eq:defA1}
&\langle	\mathcal{A}^{(1)} (\vv,\tS), \Phi \rangle:= 
	\langle \t\vv, \Phi\rangle
	+\int_\Om ( \vv\cdot \nabla \vv)  \cdot \Phi \;\ddd x
	+ \left( \eta\tS + 2 \mu (\nabla \vv)_{\sym} \right ):\nabla \Phi  \;\ddd x
	- \langle \f f, \Phi \rangle,
\\\label{eq:defA2}
&\begin{aligned}
\langle	\mathcal{A}^{(2)} (\vv,\tS),\mathbb{T} \rangle
:= 
\langle \t \tS, \mathbb{T} \rangle
+\int_\Om\big(\vv\cdot\nabla\tS + \tS (\nabla \vv)_{\skw} &- (\nabla \vv)_{\skw} \tS  \big):\mathbb{T} 
\\
&
+  \gamma \nabla  \tS\dreidots\nabla\mathbb{T}-\eta(\nabla \vv)_{\sym}:\mathbb{T}\;\ddd x.
\end{aligned}
\end{align}
\end{subequations}
Note that the interpolation inequality of Lemma~\ref{lem:ineq} shows that 
\[
\forall (\f v , \Se)\in \Xb, \, (\vv,\tS)\in\Yb: \quad
	\left \langle \mathcal{A} (\vv,\tS) ,\begin{pmatrix}
		\f v \\ \Se
	\end{pmatrix}\right \rangle  \in \LRloc{1}([0,T)),
\]
where $\Xb$ and $\Yb$ were introduced in \eqref{eq:weaksolspace} and \eqref{eq:testfctspace}.

Now we can state the above-mentioned relative energy inequality. 
Actually, we derive a family of relative energy inequalities 
where the regularity weight $\mathcal{K}$ is not fixed in advance.
Observe that the choice of $\mathcal{K}$ influences the class of
admissible test functions,
which must belong to the set $D_{\mathcal K}$ defined in \eqref{eq:DK.def}.

\begin{proposition}\label{prop:relen}
Let $(\f v , \Se)\in \Xb$ be a generalized solution according to Definition~\ref{def:gensol},
and let $\mathcal{K}:\statespace\to[0,\infty]$ be a given functional.
 Then $(\f v , \Se)$ fulfills the relative energy inequality 
\begin{multline}
\mathcal{R} (\f v(t) , \Se(t)| \vv(t),\tS(t)) \\
+\int_0^t  \left (\mathcal{W}^{(\mathcal{K})} (\f v , \Se| \vv,\tS ) + \mathcal{P}(\Se)-\mathcal{P}(\tS) + \left \langle \mathcal{A} (\vv,\tS) ,\begin{pmatrix}
\f v - \vv \\ \Se-\tS 
\end{pmatrix}\right \rangle \right ) \e^{\int_s^t \mathcal{K}(\vv,\tS) \de \tau} \de s\\
\leq \mathcal{R} (\f v_0 , \Se_0| \vv(0),\tS(0))  \e^{\int_0^t \mathcal{K}(\vv,\tS) \de s }  \, \label{relen}
\end{multline}
for
a.e.~$t\in(0,T)$ and all $(\vv,\tS)\in D_\mathcal{K}\cap\Yb$, 
where $\mathcal{W}:=\mathcal{W}^{(\mathcal{K})}$
denotes the relative dissipation-like quantity
  \begin{equation}
  \begin{split}
    \mathcal{W}^{(\mathcal{K})} (\f v,\Se|\vv,\tS) :={}& \mu \| \nabla \f v - \nabla \vv\|_{L^2(\Omega)}^2+ \gamma \| \nabla \Se-\nabla \tS \|_{L^2(\Omega)}^2
  \\
&   -\int_\Om (( \f v - \vv)\cdot\nabla (\f v - \vv)) \cdot \vv
+ (( \f v - \vv)\cdot\nabla (\Se-\tS)):\tS \;\ddd x
\\
&- \int_\Om \left ((\Se-\tS)(\nabla \f v -\nabla \vv )_{\skw} - (\nabla \f v - \nabla \vv )_{\skw} (\Se-\tS ) \right): \tS \;\ddd x
\\
& + \mathcal{K}(\vv,\tS) \mathcal{R}(\f v, \Se|\vv, \tS),
  \end{split}
\label{W}
  \end{equation}
which depends on the regularity weight $\mathcal{K}$. 
\end{proposition}
Observe that by Lemma~\ref{lem:ineq} the quantity $\mathcal{W}^{(\mathcal{K})}$ satisfies  
$\mathcal{W}^{(\mathcal{K})} (\f v,\Se|\vv,\tS)\in L^1_\loc([0,T))$
for all $(\f v,\Se),(\vv,\tS)$ as above.

\begin{proof}[Proof of Proposition~\ref{prop:relen}] Let $(\vv,\tS)\in D_\mathcal{K}\cap\Yb$ and $t\in(0,T)$.  The reasoning below is valid for a.e.\ $t\in(0,T)$.

	If $\int_0^t\mathcal{P}(\tS)\dtau=+\infty$, the asserted inequality is trivially satisfied.	Thus, we may henceforth assume that $\mathcal{P}(\tS)\in L^1(0,t)$. 		
Since $(\f v , \Se)$ is a generalized solution, it fulfills \eqref{est:enineq.v} and \eqref{est:varin0T}.
Given $\phi\in \tilde{\C} ([0,t]),$
we transform these two inequalities according to Lemma~\ref{lem:invar} into their weak formulation in time and add the weak form~\eqref{eq:weaksol.v} of the Navier--Stokes equations with the test function $\Phi = -\phi \vv$
(which is admissible as a test function due to a classical approximation argument using Lemma \ref{lem:ineq}) to obtain
\begin{multline}
-\int_0^t \phi' \left (\frac{1}{2} \| \Se - \tS \|_{L^2(\Omega)}^2 + \frac{1}{2}\| \f v \|_{L^2(\Omega)}^2 - \int_\Omega \f v \cdot \vv \dx \right ) \ddd\tau
\\
+ \int_0^t \phi \left ( \int_\Omega  \mu \nabla \f v : \left ( \nabla \f v - \nabla \vv \right ) + \gamma \nabla \Se \dreidots \left ( \nabla \Se - \nabla \tS \right ) \dx + \potential (\Se) - \potential (\tS) \right )\ddd\tau
\\
+ \int_0^t\phi  \int_\Omega -(\f v\cdot\nabla) \f v \cdot \vv + \eta \Se :\left  (\nabla \f v - \nabla \vv \right ) - (\f v \cdot \nabla) \Se:\tS - \left ( \Se (\nabla \f v )_{\skw} - (\nabla \f v)_{\skw} \Se  \right ): \tS  \dx \,\ddd\tau \\\
+\int_0^t \phi \left ( \langle\pt\vv , \f v\rangle - \int_\Omega \eta \sy{v} : ( \Se-\tS) \dx + \langle \f f , \f v - \vv \rangle\right )  \ddd\tau
\\
-  \left ( \frac{1}{2}\| \Se_0-\tS(0) \|_{L^2(\Omega)}^2 + \frac{1}{2}\| \f v_0\|_{L^2(\Omega)}^2 - \int_\Omega \f v_0 \cdot \vv(0) \dx \right ) 
\leq 0 \,.\label{weaksum}
\end{multline}
Observing that
\begin{align}
\begin{split}
&\left\langle \mathcal{A}(\vv,\tS), \begin{pmatrix}
\f v -\vv   \\ \Se -\tS 
\end{pmatrix}\right \rangle \\
&\ 
= \langle\t \vv ,  \f v\rangle +\int_ \Omega- \frac{1}{2}\t |\vv|^2  + (\vv\cdot \nabla )\vv \cdot \f v   + \left (\eta  \tS + 2 \mu (\nabla \vv)_{\sym} \right ) : (\nabla \f v-\nabla \vv)   \dx - \langle \f f, \f v -\vv  \rangle  \\
&\qquad
+\int_\Omega \t \tS : (\Se-\tS)+ ( \vv \cdot \nabla) \tS: (\Se-\tS) + \left (\tS (\nabla \vv)_{\skw} - (\nabla \vv)_{\skw} \tS \right ) : (\Se-\tS)  \dx \\
&\qquad
+ \int_\Omega \gamma \nabla  \tS \dreidots \left ( \nabla \Se - \nabla \tS\right )-  \eta (\nabla \vv)_{\sym}: \left (\Se-\tS\right )  \dx
\,,
\end{split}\label{Atest}
\end{align}
we can rewrite inequality~\eqref{weaksum}, upon integration by parts in time, as
 \begin{align*}
 -\int_0^t\phi' \mathcal{R}(\f v , \Se| \vv,\tS) \dtau  &+\int_0^t\phi\left (  \mu \| \nabla \f v - \nabla \vv\|_{L^2(\Omega)}^2  + \gamma \| \nabla \Se-\nabla \tS \|_{L^2(\Omega)}^2 \right )\ddd\tau 
\\&- \int_0^t\phi \int_\Om  (( \f v - \vv ) \cdot \nabla)( \f v - \vv ) \cdot \vv +
(( \f v - \vv ) \cdot \nabla) (\Se-\tS): \tS \,\ddd x\ddd\tau 
\\&-\int_0^t\phi  \int_\Om\left ( (\Se-\tS)(\nabla \f v -\nabla \vv )_{\skw} - (\nabla \f v - \nabla \vv )_{\skw} (\Se-\tS ) \right ): \tS  \,\ddd x\,\ddd\tau
\\& +\int_0^t\phi \bigg(\potential (\Se)-\potential(\tS )+ \left \langle \mathcal{A}_\gamma(\vv,\tS) ,\begin{pmatrix}
\f v - \vv \\ \Se-\tS 
\end{pmatrix}\right \rangle \bigg) \,\ddd\tau
\\
&\leq \mathcal{R}(\f v_0 , \Se_0| \vv(0),\tS(0)) 
 \,,
 \end{align*}
 where we used the canceling of several terms (notably of the coupling term $\int_\Om\eta \sy{v} :  \Se \dx=\int_\Om\eta \nabla\f v :  \Se \dx$), the fact that $\f v$ and $\vv$ are solenoidal vector fields,
and the antisymmetry of $(\nabla \vv)_{\skw}$.
Choosing~$\phi(\tau)=\varphi(\tau)\e^{-\int_0^\tau\mathcal{K}(\vv,\tS)\de \hat\tau}$ and invoking Lemma~\ref{lem:invar} yields the asserted inequality~\eqref{relen}. 
\end{proof}

\begin{proof}[Proof of Theorem~\ref{thm:weakstrong}] 
	First, we show that the velocity field $\tvvel$
	possesses the additional regularity
	$\tvvel\in\HSRloc{1}([0,T);(\HSRNsigma{1}(\Omega))^\ast)$.
	In the same way as in the proof of Lemma \ref{lem:ineq}, we infer
	\[
	\begin{aligned}
	&\snormL{\int_0^{t}\!\!\!\int_\Omega\tvvel\cdot\grad\tvvel\cdot\Phi\dx\ds}
	=\snormL{\int_0^{t}\!\!\!\int_\Omega\tvvel\cdot\grad\Phi\cdot\tvvel\dx\ds}
	\\
	&\qquad\leq C \bp{ \| \nabla \tvvel \|_{\LR{2}(0,t;\LR{2}(\Omega))} ^2 
	+\|  \tvvel\| _{\LR{\infty}(0,t;\LR{2}(\Omega))}^2
	+ \| \tvvel \|_{\LR{s}(0,t;\LR{r}(\Omega))}^2 } 
	\| \nabla \Phi \| _{\LR{2}(0,t;\LR{2}(\Omega))}
	\end{aligned}
		\]
	for all $\Phi\in\CRcisigma\np{\Omega\times(0,t)}$ and all $t\in(0,T)$.
	Invoking the weak formulation \eqref{eq:weaksol.v},
	we deduce
	\[
	\begin{aligned}
		&\snormL{\int_0^{t}\!\!\!\int_\Omega
		\tvvel\cdot\pt\Phi\dx\ds}
		=\snormL{\int_0^{t}\!\!\!
		\int_\Omega\bp{(\tvvel\cdot\nabla)\tvvel\cdot\Phi
		+\eta \tStens:\grad\Phi
		+\mu\grad\tvvel:\grad\Phi
		}\dx\ds
		-\int_0^{t}\langle\vf,\Phi\rangle\ds}
		\\
		&\qquad\leq
		C \bp{ \| \nabla \tvvel \|_{\LR{2}(0,t;\LR{2}(\Omega))} ^2
		+\|  \tvvel\| _{\LR{\infty}(0,t;\LR{2}(\Omega))}^2 
		+ \| \tvvel \|_{\LR{s}(0,t;\LR{r}(\Omega))}^{2}
		\\
		&\qquad\qquad\quad
		+\norm{\tStens}_{\LR{2}(0,t;\LR{2}(\Omega))}
		+\norm{\tvvel}_{\LR{2}(0,t;\HSR{1}(\Omega))}
		+\norm{\vf}_{\LR{2}(0,t;\HSR{-1}(\Omega))} } 
		\norm{\Phi}_{\LR{2}(0,t;\HSR{1}(\Omega))}.
	\end{aligned}
	\]
	This shows $\tvvel\in\HSRloc{1}([0,T);(\HSRNsigma{1}(\Omega))^\ast)$.
	Hence we have $\np{\tvvel,\tStens}\in\Yb$,
	so that $\np{\vvel,\Stens}$, $\np{\tvvel,\tStens}$ satisfy 
	the relative energy inequality \eqref{relen}
	by Proposition~\ref{prop:relen}.
	
	For the derivation of \eqref{est:stability} from \eqref{relen}
	it suffices to show 
	that for a.e.\ $t\in(0,T)$ we have
	\begin{equation}\label{eq:104}
		\begin{aligned}
		\mathcal{W}^{(\mathcal{K})} (\f v , \Se| \vv,\tS )+\mathcal{P}(\Stens)-\mathcal{P}(\tS) 
		&+ \left \langle \mathcal{A}(\vv,\tS) ,\begin{pmatrix}
			\f v - \vv \\ \Se-\tS 
		\end{pmatrix}\right \rangle
		\\
		&\ge\frac{\mu}{2} \| \nabla \f v - \nabla \vv\|_{L^2(\Omega)}^2
	+ \frac{\gamma}{2} \| \nabla \Se-\nabla \tS \|_{L^2(\Omega)}^2
		\end{aligned}		
	\end{equation}
for $\mathcal K$ given by \eqref{eq:Kstability}.
	To verify inequality~\eqref{eq:104}, we first show, for a.e.~$t\in(0,T)$, the two relations
	\begin{align}
		\big\langle \mathcal{A}^{(1)}(\vv,\tS),	\f v - \vv \big\rangle &= 0,\label{eq:101}
		\\\mathcal{P}(\Stens)-\mathcal{P}(\tS) +\big\langle \mathcal{A}^{(2)}(\vv,\tS),	\Se-\tS \big\rangle &\ge 0.\label{eq:102}
	\end{align}
	Note that, as a consequence of the weak solution property of $(\vv,\tS)$ and the extra regularity of $\vv$, for a.e.~$t\in(0,T)$ we have the equality
	\begin{align*}
		 \mathcal{A}^{(1)}(\vv,\tS)
		 =\t \vv + \nabla\cdot(\vv\otimes \vv) - \nabla\cdot\left ( \tS + 2 \mu (\nabla \vv)_{\sym} \right )  - \f f = 0 \qquad \text{ in }(\HSRNsigma{1}(\Omega))^*,
	\end{align*}
	which immediately gives~\eqref{eq:101}.
	Inequality~\eqref{eq:102} follows upon choosing $\mathbb{T}=\Se(t)\in\HSRdev{1}(\Omega)$ as a test function in the strong variational inequality~\eqref{eq:100} satisfied by $(\vv,\tS)$.
	We are thus left to prove that $\mathcal{W}^{(\mathcal{K})}\ge\frac{\mu}{2} \| \nabla \f v - \nabla \vv\|_{L^2(\Omega)}^2
	+ \frac{\gamma}{2} \| \nabla \Se-\nabla \tS \|_{L^2(\Omega)}^2$ if $\mathcal K$ is given by \eqref{eq:Kstability} for some sufficiently large constant $C<\infty$.
	Using Lemma~\ref{lem:ineq}, we may estimate
	\begin{equation}
	\begin{aligned}
		&\left|\int_\Om (( \f v - \vv)\cdot\nabla) (\f v - \vv) \cdot \vv
		+ (( \f v - \vv)\cdot\nabla) (\Se-\tS):\tS \;\ddd x\right|
		\\
		&\quad
		\leq \frac{\mu }{4} \| \nabla \f v - \nabla \vv\|_{\LR{2}(\Omega)}^2 + \frac{\gamma}{4}\| \nabla \Se - \nabla \tS \|_{\LR{2}(\Omega)}^2 
		+ C_1 \left ( \| \vv \|_{L^r(\Omega)}^s + \| \tS\|_{L^p(\Omega)}^q\right ) 
		\mathcal{R}(\f v,\Se|\vv,\tS) \,,
		\label{hoelderest}
		\end{aligned}
	\end{equation}
	where $(r,s)$ and $(p,q)$ fulfill~\eqref{serrin}. 
	Similarly, we estimate the terms stemming from the Zaremba--Jaumann rate by H\"older's, Gagliardo--Nirenberg's, and Young's inequality as
	\begin{align*}
	\Big |\int_\Om& \left( (\Se-\tS)(\nabla \f v -\nabla \vv )_{\skw} - (\nabla \f v - \nabla \vv )_{\skw} (\Se-\tS ) \right): \tS \;\ddd x\Big |\\
		&\leq  C  \| \nabla \f v - \nabla \vv \|_{\LR{2}(\Omega)} \| \Se-\tS \|_{\LR{2p/(p-2)}(\Omega)} \| \tS\|_{L^p(\Omega)} 
		 \\&\leq  C 
		 \| \nabla \f v - \nabla \vv \|_{\LR{2}(\Omega)} \left (\| \Se-\tS \|_{\LR{2}(\Omega)}^{1-\alpha}  \| \nabla \Se-\nabla \tS \|_{\LR{2}(\Omega)}^\alpha \| \tS\|_{L^p(\Omega)} + \| \Se-\tS \|_{\LR{2}(\Omega)}  \| \tS\|_{L^p(\Omega)} \right )
		 \\
		& \leq  \frac{\mu }{4} \| \nabla \f v - \nabla \vv\|_{\LR{2}(\Omega)}^2 + \frac{\gamma}{4}\| \nabla \Se - \nabla \tS \|_{\LR{2}(\Omega)}^2 + C_2 \left ( \| \tS\|_{L^p(\Omega)}^q + \| \tS\|_{L^p(\Omega)}^2 \right ) \| \Se-\tS \|_{\LR{2}(\Omega)}^2 \,,
	\end{align*}
	where $\alpha=3/p$.
	Hence, \eqref{eq:104} holds for 
	$\mathcal K$ as in \eqref{eq:Kstability}
	with $C=\max\set{C_1,2C_2 }$.
\end{proof}

\section{Energy-variational solutions for vanishing stress diffusion}\label{sec:ev.sol}

In this section, we investigate problem~\eqref{eqvis.0},
that is, \eqref{eqvis} for $\gamma=0$, 
where stress diffusion is not present. 
In this case, the basic energy estimates no longer provide weak sequential compactness in $L^1_\loc(\overline\Om\times[0,T))$ for the nonlinear term $\Se (\nabla \f v )_{\skw} -(\nabla \f v)_{\skw} \Se$ coming from the Zaremba--Jaumann derivative.
Thus, we have to further extend the concept of generalized solutions introduced in Def.~\ref{def:gensol}.
Here, we use a concept in the spirit of~\cite{diss,maxdiss}, which is reminiscent of the \textit{dissipative} solutions introduced by P.-L.\ Lions in~\cite{Lions_incompressible_1996}.
In the literature, the term \textit{dissipative} is, however, employed for various kinds of paradigms,
and here we prefer the term \textit{energy-variational}
 since the basis of our solution concept is a relative energy inequality (similar to~\eqref{relen}), which can be interpreted as a variation of the energy-dissipation mechanism  with respect to functions in the domain $D_{\mathcal{K}}$
 of the regularity weight~$\mathcal{K}$. 
We show existence of energy-variational solutions and 
derive some important properties inherent in this solution concept.

\subsection{The concept of energy-variational solutions}
We introduce the notation
\[
	\Xbzero 
	\coloneqq \LHT\times  \LRloc{\infty}\np{[0,T);\LRdev{2}(\Omega)},
	\qquad 
	\Ybzero
	\coloneqq \bigcup_{s\in(2,\infty)}\YT^s\times Z_{T,0}, 
\]
where
\[
	Z_{T,0}
	\coloneqq 	\WSRloc{1}{1}([0,T);\LR{2}_\delta(\Omega))	\cap \LRloc{2}([0,T);W^{1,3}(\Omega)^{3\times3}\cap L^{\infty}(\Omega)^{3\times 3}).
\]
As explained above, the notion of energy-variational solutions to \eqref{eqvis.0} is based on 
an adaptation of the relative energy inequality \eqref{relen} to the case $\gamma=0$.
To emphasize the $\gamma$-dependence in the quantities appearing in~\eqref{relen}, 
we set 
$\mathcal{A}^{(2)}_\gamma=\mathcal{A}^{(2)}$,
$\mathcal{A}_\gamma=\mathcal{A}=\bp{\mathcal{A}^{(1)},\mathcal{A}^{(2)}_\gamma}$ for $\gamma\geq 0$,
and
$\mathcal{W}_\gamma=\mathcal{W}^{(\mathcal{K})}_\gamma=\mathcal{W}^{(\mathcal{K})}$
for $\gamma>0$,
where $\mathcal{A}^{(2)}$
and $\mathcal{W}^{(\mathcal{K})}$
are given in \eqref{eq:defA2} and \eqref{W}, respectively.
Observe that in the limiting case $\gamma=0$ 
the quantity $\mathcal{W}^{(\mathcal{K})}$ is not well defined by \eqref{W}
for general $\np{\vvel,\Stens}\in\Xbzero$ and $\np{\tvvel,\tStens}\in\Ybzero$ 
since $\grad\Stens$ appears in the advective term.
Therefore, we formally integrate by parts and define $\mathcal{W}_0=\mathcal{W}^{(\mathcal{K})}_0$ by
\begin{equation}
	\begin{split}
		\mathcal{W}^{(\mathcal{K})}_0 (\f v,\Se|\vv,\tS) :={}& \mu \| \nabla \f v - \nabla \vv\|_{L^2(\Omega)}^2
		\\
		&   -\int_\Om (( \f v - \vv)\cdot\nabla (\f v - \vv)) \cdot \vv
		- (\Se-\tS)\otimes( \f v - \vv)\dreidots\nabla\tS \;\ddd x
		\\
		&- \int_\Om \left ((\Se-\tS)(\nabla \f v -\nabla \vv )_{\skw} - (\nabla \f v - \nabla \vv )_{\skw} (\Se-\tS ) \right): \tS \;\ddd x
		\\
		& + \mathcal{K}(\vv,\tS) \mathcal{R}(\f v, \Se|\vv, \tS).
	\end{split}
	\label{W.0}
\end{equation}
We now define energy-variational solutions by imposing a relative energy inequality 
analogous to~\eqref{relen}. 
Recall that the regularity weight $\mathcal K$ was not fixed in \eqref{relen},
which is reflected in the following solution concept.
We emphasize though that the definition to follow is only meaningful for sufficiently ``nice'' weights $\mathcal{K}$. 

\begin{definition}\label{def:envar}
	Let $\mathcal{K}:\statespace\to [0,\infty]$ with $\mathcal{K}( \f 0, \mathbb{O})= 0$.
	A tuple $(\f v , \Se)\in \Xbzero$ is called an energy-variational solution to~\eqref{eqvis.0} of type $\mathcal{K}$
if 
		the relative energy inequality 
		\begin{multline}
			\mathcal{R} (\f v(t) , \Se(t)| \vv(t),\tS(t)) 
			\\+\int_0^t  \left (\mathcal{W}_0^{(\mathcal K)} (\f v , \Se| \vv,\tS ) + \mathcal{P}(\Se)-\mathcal{P}(\tS)+ \left \langle \mathcal{A}_0 (\vv,\tS) ,\begin{pmatrix}
				\f v - \vv \\ \Se-\tS 
			\end{pmatrix}\right \rangle \right ) \e^{\int_s^t \mathcal{K}(\vv,\tS) \de \tau} \de s\\  \leq \mathcal{R} (\f v_0 , \Se_0| \vv(0),\tS(0))  \e^{\int_0^t \mathcal{K}(\vv,\tS) \de s }  \, \label{relen.}
		\end{multline}
	is fulfilled for a.e.~$t\in (0,T)$ and all $(\vv,\tS)\in D_\mathcal{K}\cap\Ybzero$. 
	As before, we call the function~$\mathcal{K}$ the regularity weight. 
\end{definition}
\begin{remark}[Choice of the regularity weight]
While the assumption $\mathcal{K}( \f 0, \mathbb{O})= 0$ is not needed to define such a solution concept,
it is a natural hypothesis which ensures the classical energy inequality, \textit{i.e.,} inequality~\eqref{est:enineq} with $\gamma=0$. Note that this inequality directly follows by letting 
 $(\vv,\tS) \equiv ( \f 0, \mathbb{O})$ in~\eqref{relen.}, which is admissible when $\mathcal{K}( \f 0, \mathbb{O})= 0$.
This in turn guarantees that  $ \mathcal{P}(\Se)\in L^1_\loc([0,T))$ and hence, that the terms in~\eqref{relen.} are well defined for all $(\vv,\tS)\in D_\mathcal{K}\cap\Ybzero$. 
Observe that standard interpolation estimates 
show that 
\[
	\mathcal{W}_0^{(\mathcal K)} (\f v , \Se| \vv,\tS ),\ \left \langle \mathcal{A}_0 (\vv,\tS) ,\begin{pmatrix}
		\f v - \vv \\ \Se-\tS 
	\end{pmatrix} \right \rangle \in L^1_\loc([0,T))
\]
for $(\f v , \Se)\in \Xbzero$ and $(\vv,\tS)\in \Ybzero\cap D_\mathcal{K}$, 
see also the estimates in the proof of Theorem \ref{thm:envar} below.

Notice that the classical energy inequality encodes some basic information concerning the 
long-time behavior of solutions, cf.~\cite[Thm.~6.2]{diss}.

\end{remark}

\begin{remark}[Comparison to previous formulations]\label{rem:formulation}
The above formulation differs from previous notions of dissipative solutions (cf.~\cite{Lions_incompressible_1996}), which were based on similar relative energy inequalities.
In those definitions, the 
regularity weight $\mathcal{K}$ was fixed and chosen such that the last three lines of~\eqref{W.0} are non-negative. 
Additionally, these terms are usually just estimated from below by zero.
By 
keeping the mentioned terms in the inequality~\eqref{relen.}, the present solution concept is more selective than previous versions. 
\end{remark}

\subsection{Global existence of energy-variational solutions}\label{subsec:existence.envarsol}

In Proposition \ref{prop:relen} we have seen that the generalized solutions for $\gamma>0$
satisfy the relative energy inequality~\eqref{relen}.
For suitable regularity weights $\mathcal K$ we may perform the limit $\gamma\to0$ in this inequality to obtain the following existence result.

\begin{theorem}[Existence of energy-variational solutions]\label{thm:envar}
	Under the assumptions stated in Hypothesis \ref{hypo},
	 there exists an energy-variational solution to the initial-boundary value problem~\eqref{eqvis.0} of type $\mathcal{K}$ given by
	 \begin{equation}\label{KStilde}
	 \mathcal{K}(\tS):=
	 C\bp{\|\tS\|_{L^\infty(\Omega)}^2+\|\nabla \tS\|_{L^3(\Omega)}^2}
	 \end{equation}
	 for a suitable constant $C=C(\Omega,\mu)>0$.
	 
	 Moreover, any such energy-variational solution satisfies the Navier--Stokes component \eqref{eqvis1.0}
	 in the weak sense,
	 that is,
	 \eqref{eq:weaksol.v} is satisfied for all $\Phi\in\CRcisigma(\Omega\times[0,T))$.
\end{theorem}

\begin{proof}
In order to prove the existence of energy-variational solutions, for each $\gamma>0$ 
we consider a generalized solution $(\f v_\gamma  , \Se_\gamma  )\in \Xb$ 
to \eqref{eqvis},
which exists due to Theorem~\ref{thm:gensol}. 
By Remark \ref{rem:enineq.gensol} the generalized solution $(\vgam , \Sgam )$ 
fulfills the energy inequality
\begin{multline}
\frac{1}{2}\norm{\vgam(t)}_{L^2(\Omega)}^2 + \frac{1}{2} \norm{\Sgam (t)}_{L^2(\Omega)}^2 
	+\mu\int_0^{t}\norm{\grad\vgam}_{L^2(\Omega)}^2\ds+ \gamma \int_0^{t} \norm{\grad \Sgam}_{L^2(\Omega)}^2 \ds + \int_0^{t} \potential{(\Sgam)} \ds 
\\
	\leq\frac{1}{2} \norm{\vvel_0}_{L^2(\Omega)}^2
	+\frac{1}{2}\norm{\Stens_0}_{L^2(\Omega)}^2
	+\int_0^{t} \langle \f f , \vgam \rangle \ds\label{enin}
\end{multline}
for a.e.~$t\in (0,T)$. Additionally, we may find a $\gamma$-independent bound on the time derivative of the velocity field. 
Indeed, since $\np{\vvel_\gamma,\Stens_\gamma}$ satisfies the weak formulation~\eqref{eq:weaksol.v}, 
for every $\Phi \in \CRcisigma\np{\OmT}$ we can estimate 
\begin{align*}
&\left |\int_0^t ( \vgam , \t\Phi ) \ds\right |
= \left |\int_0^t\langle \f f , \Phi \rangle + \left ( \vgam \otimes \vgam , \nabla \Phi \right  ) - \mu \left ( \nabla \vgam , \nabla \Phi \right ) - \eta \left ( \nabla \Sgam , \nabla \Phi \right ) \ds\right |
\\
&\qquad
\leq C  \int_0^t \left (\| \f f \|_{(\HSRNsigma{1})^*} + \| \vgam \|_{L^4(\Omega)}^2 + \mu \| \nabla \vgam \|_{L^2(\Omega)} + \eta \| \grad\Sgam \|_{L^2(\Omega)}\right )  \| \nabla \Phi \|_{L^2(\Omega)}\ds.
\end{align*}
Using the Ladyzhenskaya inequality $ \| \vgam \|_{L^4(\Omega)} \leq C \| \vgam\|_{L^2(\Omega)} ^{1/4} \| \nabla \vgam \|_{L^2(\Omega)}^{3/4} $, we find that 
\begin{multline*}
 \| \t \vgam \|_{L^{4/3}(0,t;(\HSRNsigma{1})^*)} ^{4/3} = \int_0^{t} \| \t \vgam \|_{(\HSRNsigma{1})^*} ^{4/3} \ds\\
  \leq C \int_0^{t}  \| \f f \|_{(\HSRNsigma{1})^*}^{4/3} +( \| \vgam\|_{L^2(\Omega)} ^{1/2} \| \nabla \vgam \|_{L^2(\Omega)}^{3/2})^{4/3}  + \mu \| \nabla \vgam \|_{L^2(\Omega)}^{4/3} + \eta \| \Sgam \|_{L^2(\Omega)}^{4/3} \ds
 \\
 \leq C\left (  \| \f f \|_{L^2(0,t;(\HSRNsigma{1})^*)}^{4/3} + \| \vgam\|_{L^\infty(0,t;L^2(\Omega) )}^{2/3} \| \vgam\|_{L^2(0,t;\HSRNsigma{1})}^2 + \| \vgam\|_{L^2(0,t;\HSRNsigma{1})}^{4/3} + \| \Sgam\|_{L^2(0,t;L^2(\Omega))}^{4/3} \right ) \,.
\end{multline*}
By the aforementioned bounds and a classical diagonalization argument 
(see \cite[Proof of Theorem 3.1.1]{Sohr_NSE_2001} for example), 
we infer the existence of a limit function $\np{\vvel,\Stens}$ such that
for each $t\in(0,T)$ we have
\begin{subequations}\label{conv}
\begin{align}
	\vgam&\rightharpoonup\f v  && \tin \LR{2}(0,t;\HSR{1}(\Omega)^3), \label{conv:vweak}
	\\
	\vgam&\overset{\ast}{\rightharpoonup}\f v  && \tin \LR{\infty}(0,t;\LRsigma{2}(\Omega))\label{conv:vweakstar},
	\\
	\vgam&\to\f v  && \tin \LR{2}(0,t;\LR{2}(\Omega)^3),\label{conv:vstrong}
	\\
	\Sgam &\overset{\ast}{\rightharpoonup} \Se  && \tin \LR{\infty}(0,t;\LRdev{2}(\Omega))\label{conv:S}
\end{align}
\end{subequations}
for an appropriately chosen subsequence $\gamma \ra 0$. 
Next we conclude that via these convergences, we may pass to the limit in 
the relative energy inequality~\eqref{relen}. 

Proposition~\ref{prop:relen} guarantees that~\eqref{relen} holds for   
all non-negative functionals $\mathcal{K}:\statespace\to[0,\infty]$, all $(\vv,\tS)\in D_\mathcal{K}\cap\Yb$
and a.e.~$t\in(0,T)$. 
We may reformulate this inequality using Lemma~\ref{lem:invar} and observing that $\mathcal{W}_\gamma \geq \mathcal{W}_0$. This leads to
\begin{equation}
\begin{aligned}
&- \int_0^t \phi' \mathcal{R}(\vgam ,\Sgam| \vv,\tS ) \ds + \int_0^t \phi \mathcal{W}_0 (\vgam ,\Sgam| \vv,\tS ) \e^{\int_s^t \mathcal{K}(\vv,\tS) \de \tau} \de s 
\\
&\quad
+\int_0^t \phi \bp{\mathcal{P}(\Sgam)-\mathcal{P}(\tS)} \e^{\int_s^t \mathcal{K}(\vv,\tS) \de \tau} \de s
+ \int_0^t \phi \left \langle \mathcal{A}_{\gamma} ( \vv,\tS) , \begin{pmatrix}
\vgam-\vv \\ \Sgam-\tS 
\end{pmatrix}\right \rangle \e^{\int_s^t \mathcal{K}(\vv,\tS) \de \tau} \de s 
\\
&\qquad\qquad\qquad\qquad\qquad\qquad\qquad\qquad\qquad\qquad
\leq \mathcal{R}(\f v_0 ,\Se_0 | \vv(0) , \tS(0)) \e^{\int_0^t\mathcal{K}(\vv,\tS)\de s }
\end{aligned}
\label{relen0weak}
\end{equation}
for all $\phi \in \tilde{\C} ([0,T])$. 
We now show that, when choosing $\mathcal K$ as in \eqref{KStilde},
we may pass to the limit $\gamma\to0$ in each of the terms on the left-hand side 
of \eqref{relen0weak}
via the convergence properties from \eqref{conv}.

Since $\phi'\le0$, the first term in~\eqref{relen0weak} is  convex and continuous on $L^2(\Omega\times(0,t))$ as a function of $(\vgam ,\Sgam)$. Hence, it is weakly
lower semicontinuous, and its convergence follows from~\eqref{conv:vstrong} and~\eqref{conv:S}.

In order to pass to the limit in the second integral term in \eqref{relen0weak}, we split
$$\mathcal{W}_0 (\vgam ,\Sgam| \vv,\tS ) = \mathcal{C}( \vgam,\vv)+
\mathcal{Q}(\vgam ,\Sgam| \vv,\tS ),$$
where $\mathcal{C}( \vvel,\vv):=-\int_\Om(( \f v - \vv)\cdot\nabla (\f v - \vv)) \cdot \vv \,\dx$ \;and
\begin{equation*}
	\begin{split}
		\mathcal{Q}(\f v ,\Se| \vv,\tS ) :={}& \mu \| \nabla \f v - \nabla \vv\|_{L^2(\Omega)}^2 
		+ \mathcal{K}(\tS) \mathcal{R}(\f v, \Se|\vv, \tS)
		+ \int_\Om 
		(\Se-\tS)\otimes( \f v - \vv)\dreidots\nabla\tS \;\ddd x		
		\\&	- \int_\Om \left ((\Se-\tS)(\nabla \f v {-}\nabla \vv )_{\skw} - (\nabla \f v {-} \nabla \vv )_{\skw} (\Se-\tS ) \right): \tS \;\ddd x.
	\end{split}
\end{equation*}
The convergence of the part of the integral involving the term $\mathcal{C}( \vgam,\vv)$ follows as usual 
since the strong convergence~\eqref{conv:vstrong} combined with~\eqref{conv:vweak} implies the weak convergence of $\f v_\gamma\cdot\nabla \f v_\gamma$.
Concerning the remaining part involving  $\mathcal{Q}(\vgam ,\Sgam| \vv,\tS )$, we assert that 
(for fixed $\vv\in\HSRNsigma{1}(\Omega)$ and $\tS\in W^{1,3}(\Om)\cap L^\infty(\Om)$) the continuous functional
$Q:\HSRNsigma{1}(\Omega)\times\LRdev{2}(\Om)\to\mathbb{R},$ $Q(\f v ,\Se):=\mathcal{Q}(\f v ,\Se| \vv,\tS )$ is convex provided the constant $C=C(\mu,\Om)$ in the definition of $\mathcal{K}(\tS)$ is chosen large enough.  Since $Q$ is the sum of a quadratic form $Q_1$ and an affine part, it suffices to show that $Q\ge0$, as this implies the non-negativity of $Q_1$ and hence its convexity (cf.~\cite[Prop.~3.71]{BS_2000}).
To show the non-negativity of $Q$ (and its continuity) we use the Sobolev--Poincar\'e inequality 
$\norm{\tvvel}_{\LR{6}(\Omega)}\leq C\norm{\grad\tvvel}_{L^2(\Omega)}$
and estimate, choosing $C$ sufficiently large,
\begin{align*}
	&	\left|\int_\Om (\Se-\tS)\otimes( \f v - \vv)\dreidots\nabla\tS \;\ddd x	\right|
	\le \frac{\mu}{2} \| \nabla \f v - \nabla \vv\|_{L^2(\Omega)}^2
	+ C\| \nabla \tS \|_{L^3(\Omega)}^2 \mathcal{R}(\f v, \Se|\vv,\tS).
\end{align*}
Combined with the bound
\begin{align*}
	 &\left|\int_\Om \big ((\Se-\tS)(\nabla \f v {-}\nabla \vv )_{\skw} - (\nabla \f v {-}\nabla \vv )_{\skw} (\Se-\tS ) \big): \tS \;\ddd x\right|
	\\&\hspace{.3\linewidth}\le \frac{\mu}{2} \| \nabla \f v - \nabla \vv\|_{L^2(\Omega)}^2 
	+ C \| \tS \|_{L^\infty(\Omega)}^2 
	\frac{1}{2} \| \Se-\tS \|_{L^2(\Omega)}^2,
\end{align*}
this shows that $Q$ is non-negative and continuous. Hence, we may use weak lower semicontinuity to take the limit $\gamma\to0$ in the integral term 
$ \int_0^t \phi \mathcal{Q} (\vgam ,\Sgam| \vv,\tS ) \e^{\int_s^t \mathcal{K}(\vv,\tS) \de \tau} \de s$.

For the third term in \eqref{relen0weak},
note that the induced functional 
$\mathbb{T}\mapsto\int_0^t \phi \mathcal{P}(\mathbb{T}) \e^{\int_s^t \mathcal{K}(\vv,\tS) \de \tau} \de s$
is convex and, by Fatou's lemma,  lower semicontinuous on $\LR{2}(0,t;\LR{2}(\Omega))$.
This implies its weak lower semicontinuity, whence
\[
\liminf_{\gamma\to0}
\int_0^t \phi \mathcal{P}(\Sgam) \e^{\int_s^t \mathcal{K}(\vv,\tS) \de \tau} \de s
\geq\int_0^t \phi 
\mathcal{P}(\Stens) \e^{\int_s^t \mathcal{K}(\vv,\tS) \de \tau} \de s.
\]

For the term involving the operator $\mathcal{A}_\gamma$, 
we first estimate the $\gamma$-dependent term as
\[
\begin{aligned}
&\int_0^t \gamma \bp{\nabla \tS \,, \nabla ( \Sgam -\tS ) }\ds \\ 
&\qquad
\leq \sqrt{\gamma} \left (\int_0^t\| \nabla \tS \|_{L^2(\Omega)}\ds\right )^{1/2} \left ( \int_0^t \gamma\| \nabla\Sgam\|_{L^2(\Omega)}^2 \ds\right )^{1/2} \!\!+ \gamma \| \tS \|_{\LR{2}(0,t;\HSR{1}(\Omega))}^2
\end{aligned}
\]
for all $t\in(0,T)$.
Due to the bound~\eqref{enin} on $ \int_0^t \gamma\| \nabla\Sgam\|_{L^2(\Omega)}^2 \ds$ and the regularity of~$\tS$, the right-hand side of the last inequality tends to zero as $\gamma \ra 0$. 
Since the remaining terms do not explicitly depend on $\gamma$
and $(\vgam,\Sgam)$ occurs at most linearly,
we conclude convergence of the last term on the left-hand side of \eqref{relen0weak}.

In total, we can pass to the limit inferior with $\gamma \ra 0$ in~\eqref{relen0weak},
which implies inequality~\eqref{relen.} by invoking Lemma~\ref{lem:invar}.
Therefore, $\np{\vvel,\Stens}$ is an energy-variational solution to \eqref{eqvis.0}.

Since $\mathcal{K}(\vv,\mathbb{O})=0$ for all $\vv\in \CRcisigma\np{\Om}$,
we further conclude from Proposition \ref{prop:envarweak} below that 
\eqref{eqvis1.0} is satisfied in the weak sense. 
\end{proof}

\begin{remark}
\label{rem:weaklystarclosed}
An argument similar to the proof of Theorem \ref{thm:envar} shows 
that the set of energy-variational solutions of type $\mathcal{K}$ given by~\eqref{KStilde} is
weakly-$\ast$ sequentially compact in $\Xbzero$ in the following sense:
For every sequence of energy-variational solutions $(\vvel_j,\Stens_j)$
there exists an energy-variational solution $\np{\vvel,\Stens}\in\Xbzero$
and a subsequence (also denoted by $(\vvel_j,\Stens_j)$)
with
\[
	\vvel_j\rightharpoonup\f v  
	\ \tin \LR{2}(0,t;\HSR{1}(\Omega)^3),
	\qquad
	\vvel_j\overset{\ast}{\rightharpoonup}\f v  
	\ \tin \LR{\infty}(0,t;\LRsigma{2}(\Omega)),
	\qquad
	\Stens_j \overset{\ast}{\rightharpoonup} \Se  
	\ \tin \LR{\infty}(0,t;\LRdev{2}(\Omega))
\]
for all $t\in(0,T)$.
Indeed, a sequence $\np{\vvel_j,\Stens_j}$
of energy-variational solutions of type $\mathcal{K}$
satisfies the relative energy inequality \eqref{relen.}
with $\np{\tvvel,\tStens}=(\f 0, \zerotens)$, that is,
the energy inequality
\[
\begin{aligned}
\frac{1}{2}\norm{\vvel_j(t)}_{\LR{2}(\Omega)}^2
+\frac{1}{2}\norm{\Stens_j(t)}_{\LR{2}(\Omega)}^2
&+\int_0^t \mu\norm{\grad\vvel_j}_{\LR{2}(\Omega)}^2+\potential(\Stens_j)\ds
\\
&\qquad
\leq\frac{1}{2}\norm{\vvel_0}_{\LR{2}(\Omega)}^2
+\frac{1}{2}\norm{\Stens_0}_{\LR{2}(\Omega)}^2
+\int_0^t \langle \f f,\vvel_j\rangle\ds
\end{aligned}
\]
for a.a.~$t\in(0,T)$. Moreover, Proposition~\ref{prop:envarweak} below ensures that the Navier--Stokes component holds in the weak sense.
In the same way as above, 
we further deduce that $(\partial_t\vvel_j)$ is bounded in 
$L^{4/3}(0,t;(\HSRNsigma{1})^*)$ for all $t\in(0,T)$,
and we conclude
the asserted convergence properties as well as
$\vvel_j\to\vvel$ in $\LR{2}(0,t;\LR{2}(\Omega))$ for all $t\in(0,T)$.
In nearly the same way as for the limit $\gamma\to0$ above,
we can now perform the limit $j\to\infty$ 
to infer that $\np{\vvel,\Stens}$ is an
energy-variational solution of type $\mathcal{K}$.
\end{remark}

\subsection{Further properties}

Here we collect general properties of energy-variational solutions.
Most importantly, we show that 
energy-variational solutions are subject to the weak formulation of 
the Navier--Stokes equations~\eqref{eqvis1.0} 
if the regularity weight $\mathcal K$ vanishes on $\CRcisigma(\Omega)\times\set{\zerotens}$, 
where $\mathbb{O}\in \LRdev{2}(\Om)$ denotes the zero tensor.
Thus, this is in particular the case for the solutions obtained in Theorem \ref{thm:envar}. 
Notice that this property is to be expected for problem~\eqref{eqvis.0} since the lack of weak sequential compactness in $L^1_\loc(\overline\Om\times[0,T))$ only occurs in the tensor component~\eqref{eqvis2.0} but not in~\eqref{eqvis1.0}.

\begin{proposition}\label{prop:envarweak}
	Suppose that the regularity weight $\mathcal{K}:\statespace\to [0,\infty] $
	satisfies 
	$\mathcal{K}(\vv,\mathbb{O})=0$ for all 
	$\vv\in \CRcisigma\np{\Om}$.
	Then every energy-variational solution of type $\mathcal{K}$ is a weak solution of the Navier--Stokes component in the sense of~\eqref{eq:weaksol.v}.
\end{proposition}

\begin{proof}	
The hypothesis on $\mathcal{K}$ implies the identity $\e^{\int_0^t\mathcal{K}(\vv, \mathbb{O}) \de s} \equiv 1$ for all $\vv\in\CRcisigma\np{\OmTzero}$. Hence,  choosing $\tS\equiv \mathbb{O}$ in~\eqref{relen.} yields the inequality
\begin{multline*}
\mathcal{R} (\f v(t) , \Se(t)| \vv(t),\mathbb{O}) +\int_0^t  \left (\mathcal{W}_0 (\f v , \Se| \vv,\mathbb{O}) +\mathcal{P}(\Se)+ \left \langle \mathcal{A}_0 (\vv,\mathbb{O}) ,\begin{pmatrix}
\f v - \vv \\\Se 
\end{pmatrix}\right \rangle \right )  \de s\\  \leq \mathcal{R} (\f v_0 , \Se_0| \vv(0),\mathbb{O})    \,
\end{multline*}
for all $\vv\in\CRcisigma\np{\OmTzero}$ and for a.e.~$t\in(0,T)$. 
Applying Lemma~\ref{lem:invar} and observing that 
\[\langle\mathcal{A}^{(1)}(\vv,\mathbb{O}),(\vvel-\vv)\rangle=\int_ \Omega \t \vv \cdot  \f v - \frac{1}{2}\t |\vv|^2  + (\vv\cdot \nabla )\vv \cdot \f v   +  2 \mu (\nabla \vv)_{\sym} : (\nabla \f v-\nabla \vv)   \dx - \langle \f f, \f v -\vv  \rangle\]
and $\langle\mathcal{A}^{(2)}_0(\vv,\mathbb{O}),\Se \rangle= -\int_ \Omega \eta\syv :\Se\;\ddd x$,
we find
\begin{equation}\label{ineqequiv}
\begin{aligned}
&-\int_0^T  \phi ' \left ( \frac{1}{2} \| \f v {-} \vv \|_{L^2(\Omega)}^2 + \frac{1}{2}\| \Se\|_{L^2(\Omega)}^2 \right ) \de t - \left ( \frac{1}{2} \| \f v_0 {-} \vv(0) \|_{L^2(\Omega)}^2 + \frac{1}{2}\| \Se_0\|_{L^2(\Omega)}^2\right ) 
\\
&\quad
+\int_0^T \phi\left ( \mu \| \nabla \f v {-} \nabla \vv \|_{L^2(\Omega)}^2 + \potential(\Se) 
+ \int_ \Omega ( \f v {-} \vv ) \otimes (\f v{-}\vv): \nabla \vv- \eta  \syv : \Se\;\ddd x \!\right ) \!\dt 
\\
&\quad\ \
+\int_0^T \phi \left ( \int_ \Omega \t \vv \cdot  \f v - \frac{1}{2}\t |\vv|^2  + (\vv\cdot \nabla )\vv \cdot \f v   +  \mu \nabla \vv  : (\nabla \f v{-}\nabla \vv)   \dx - \langle \f f, \f v {-}\vv  \rangle\right ) \dt \leq 0\,. 
\end{aligned}
\end{equation}
We now integrate by parts in time the second term in the last line of~\eqref{ineqequiv}, regroup terms, 
let $\vv = \alpha \tu$ for given $\tu\in\CRcisigma\np{\OmTzero}$ and $\alpha>0$, and multiply the inequality by $1/\alpha$ to infer
\begin{multline*}
	\int_0^T \phi' \left ( \f v , \tu\right ) \;\ddd t
+\frac{1}{\alpha}\left (- \frac{1}{2}\int_0^T \phi'\left ( \| \f v \|_{L^2(\Omega)}^2 + \| \Se \|_{L^2(\Omega)}^2\right) \dt - \frac{1}{2}\left ( \| \f v_0 \|_{L^2(\Omega)}^2 + \| \Se_0\|_{L^2(\Omega)}^2 \right )  
  \right ) 
\\
+ \frac{1}{\alpha} \int_0^T \phi\left ( \mu \| \nabla \f v \|_{L^2(\Omega)}^2 + \mathcal{P}(\Se) - \langle \f f , \f v\rangle \right )\dt 
\\
+ \int_0^T \phi\bigg(\int_\Omega \t \tu\f v - \mu \nabla \f v :\nabla \tu + (\f v \otimes \f v ) : \nabla \tu - \eta ( \nabla \tu)_{\sym}:\Se \;\ddd x + \langle \f f , \tu \rangle\bigg) \dt 
+ \left ( \f v _0 , \tu(0)\right ) \leq0 \,.
\end{multline*}
In the limit $\alpha\ra \infty$ the two terms with the prefactor $\frac{1}{\alpha}$ disappear. Thus,  appealing to Lemma~\ref{lem:invar} we deduce
\begin{align*}
\int_0^T \left(\int_\Omega \t \tu\f v - \mu \nabla \f v :\nabla \tu + (\f v \otimes \f v ) : \nabla \tu - \eta ( \nabla \tu)_{\sym}:\Se\;\ddd x  + \langle \f f , \tu \rangle\right)\dt + \left ( \f v _0 , \tu(0)\right ) \leq0 \,,
\end{align*}
where we used the fact that $\tu(\cdot,T)=0$.
Since the left-hand side of the last inequality is a linear functional of $\tu$ with the latter being allowed to vary in the linear space $\CRcisigma\np{\OmTzero}$, we infer the asserted equation~\eqref{eq:weaksol.v}.
\end{proof}

Another quite natural property is that
energy-variational solutions are monotonic in the type in the following manner.

\begin{proposition}[Monotonicity]\label{prop:type.monotonous}
If $(\vvel,\Stens)$ is an energy-variational solution of type $\mathcal K$, 
and if $\mathcal K(\tvvel,\tStens)\leq\mathcal L(\tvvel,\tStens)$ 
for all $(\tvvel,\tStens)\in D_{\mathcal L}\subset D_{\mathcal K}$ and a.a.~$t\in(0,T)$,
then $(\vvel,\Stens)$ is an energy-variational solution of type $\mathcal L$.
\end{proposition}

\begin{proof}
This can readily be seen by expressing \eqref{relen.} in a weak form with Lemma~\ref{lem:invar},
using the test function 
$\phi(s)=\varphi(s) \e^{-\int_0^s\np{\mathcal{L}(\tvvel,\tStens)-\mathcal{K}(\tvvel,\tStens)}\dtau}$
for $\varphi\in\tilde{W}((0,T))$,
and using Lemma \ref{lem:invar} again 
to return to a pointwise version of \eqref{relen.} with $\mathcal K$ replaced with $\mathcal L$.
\end{proof}

	\begin{remark}[Convex solution set]\label{rem:convex}
	By Proposition \ref{prop:type.monotonous},
	the energy-variational solution established in Theorem \ref{thm:envar} is 
	also an energy-variational solution of type $\mathcal{K}_s$ given by 
	\begin{align}
	\mathcal{K}_s(\vv, \tS) \coloneqq C \left ( \| \vv \| _{L^r(\Omega)}^s +\| \tS \|_{L^\infty(\Omega)}^2 + \| \nabla \tS \|_{L^3(\Omega)}^2 \right )\,\label{K2}
	\end{align}
	for $s\in(2,\infty)$ and $r\in(3,\infty)$ such that $2/s+3/r=1$.
	Similarly to~\eqref{hoelderest}, we observe that 
	$\mathcal{W}_0 (\f v , \Se|\vv,\tS) \geq 0$ 
	for all $(\f v, \Se)\in \mathfrak{X}_0 $ and all $(\vv,\tS) \in \mathfrak{Z}_0\cap D_{\mathcal K_s}$
	if $C>0$ is sufficiently large. 
	Since $\mathcal{W}_0^{(\mathcal{K}_s)} $ is quadratic in $(\f v,\Se)$ and non-negative, it is convex in $(\f v, \Se)$. 
	This implies that the solution set of energy-variational solutions of type $\mathcal{K}_s$ is convex. 
	Additionally, the solution set is weakly-$\ast$ closed (\textit{cf.}~Remark \ref{rem:weaklystarclosed})
	and bounded and therewith compact in the weak-$\ast$ topology of~$\mathfrak{X}_0$. 
	This convexity and compactness properties may be used to select an energy-variational solution with 
	maximal dissipation~\cite{envar}. 
	Such a selected maximally dissipative solution can not only be argued to be physically more reasonable since they minimize the energy along all energy-variational solutions, but they also turn out to be analytically favorable since the solution concept is well posed~(see~\cite{maxdiss,envar}). 
	\end{remark}

\subsection{Energy-variational solutions versus strong solutions}
In this subsection, we prove two results 
attempting to further justify the concept of energy-variational solutions.
 First, in Proposition~\ref{prop:weakstrongenvar} we show that any hypothetical strong solution is unique in the class of energy-variational solutions.
Second, in Proposition~\ref{prop:addreg} we prove that any energy-variational solution enjoying some additional regularity is a (unique) strong solution. 
Observe that both results only hold for strong solutions that belong to $D_{\mathcal K}$.
In this sense, the regularity weight $\mathcal K$ also determines the 
minimal regularity of a strong solution for comparison with an energy-variational solution.

\begin{definition}[Strong solution]\label{def:strong.0}
	We call $(\vvel,\Stens)$ 
	a strong solution of problem \eqref{eqvis.0} with initial data $(\vvel_0,\Se_0)$ 
	if $\np{\vvel,\Stens}\in\Ybzero$, $\np{\vvel(0),\Stens(0)}=\np{\vvel_0,\Stens_0}$ and 
	\begin{itemize}
		\item 
		equation~\eqref{eqvis1.0} is satisfied in the weak sense, 
		\textit{i.e.}, equation \eqref{eq:weaksol.v}
		 holds true
		for all 
		$\Phi\in\CRcisigma\np{\OmTzero}$.
		\item 
		inclusion~\eqref{eqvis2.0} is satisfied pointwise a.e.~in time, 
		\textit{i.e.}, 
		it holds
		\begin{equation}\label{eq:strongvar.0}
		-\bp{\t \Se + (\f v \cdot \nabla) \Se 
		+ \left ( \Se (\nabla \f v)_{\skw} -(\nabla \f v )_{\skw} \Se  \right ) 
		- \eta (\nabla \f v)_{\sym}}(t)
		\in\partial\potential(\Stens(t))\quad
\end{equation}
	for a.a.~$t\in (0,T)$.
	\end{itemize}
\end{definition}

We show next that 
if there exists an energy-variational solution of type $\mathcal K$ 
and a strong solution in the class $D_{\mathcal K}$,
both emanating from the same initial data,
then these solutions coincide.
This statement and its proof are parallel to the one given in Theorem \ref{thm:weakstrong} for $\gamma>0$.
In particular, we also derive a corresponding weak-strong stability estimate.

\begin{proposition}[Weak-strong uniqueness]\label{prop:weakstrongenvar}
Assume that $\mathcal K$ is a regularity weight
such that $\mathcal{W}_0=\mathcal{W}_0^{(\mathcal K)}$ defined in~\eqref{W.0} is non-negative,
\textit{i.e.}, $\mathcal{W}_0( \f v , \Se|\vv,\tS) \geq 0 $ for all $( \f v , \Se) \in \mathfrak{X}_0$ and $(\vv,\tS)\in \mathfrak{Z}_0$. 
Let $(\f v , \Se)\in \mathfrak{X}_0$ be an energy-variational solution of type $\mathcal{K}$
with initial data $(\f v_0 , \Se_0)$,
and let $( \vv , \tS) \in \mathfrak{Z}_0 \cap D_{\mathcal K}$
be a strong solution in the sense of Def.~\ref{def:strong.0}
with initial data $( \vv_0 , \tS_0)$.
Then the inequality
\begin{equation}\label{est:stabilityenvar}
	\mathcal{R} (\f v(t) , \Se(t)| \vv(t),\tS(t))
	\leq \mathcal{R} (\f v_0 , \Se_0| \vv_0,\tS_0)  \e^{\int_0^t \mathcal{K}(\vv,\tS) \de s } \,
\end{equation}
holds.
Especially, if the initial conditions coincide, \textit{i.e.,} $ (\f v_0 , \Se_0 ) = ( \vv_0, \tS_0)$, then every energy-variational solution of type $\mathcal{K}$ coincides with the (hypothetical) strong solution~$(\vv,\tS)$. 
\end{proposition}	

	\begin{proof}
	From the assumption that $\mathcal{W}_0$ is non-negative, the inequality~\eqref{relen.} also holds without this term. Similarly to the proof of Theorem~\ref{thm:weakstrong}, we observe that 
	\begin{equation*}
		\begin{aligned}
		\big\langle \mathcal{A}^{(1)}(\vv,\tS),	\f v - \vv \big\rangle &= 0,
		\\\mathcal{P}(\Stens)-\mathcal{P}(\tS) +\big\langle \mathcal{A}^{(2)}_0(\vv,\tS),	\Se-\tS \big\rangle &\ge 0
		\end{aligned}
	\end{equation*}
	for a.e.~$t\in(0,T)$. Hence, all terms in the second line of~\eqref{relen.} may be estimated from below by zero, which implies~\eqref{est:stabilityenvar}. 
	\end{proof}
	
\begin{corollary}
The energy-variational solution from Theorem~\ref{thm:envar} 
also fulfills the weak-strong uniqueneness property,
that is,
if there exists a strong solution in the sense of Def.~\ref{def:strong.0}
it coincides with the energy-variational solution from Theorem~\ref{thm:envar}.
\end{corollary}

\begin{proof}
According to Proposition~\ref{prop:type.monotonous}, an energy-variational solution of type $\mathcal{K}$ given in~\eqref{KStilde} is also an energy-variational solution solution of type $\mathcal{K}_s$ given in~\eqref{K2},
which fulfills the assumptions of Proposition~\ref{prop:weakstrongenvar}.
This ensures that the weak-strong uniqueness also holds for the smaller set of energy-variational solutions of Theorem~\ref{thm:envar},
as long as the strong solution belongs to $D_{\mathcal K_s}$ for some $s\in(2,\infty)$.
Clearly, this is the case for all strong solutions in the sense of Def.~\ref{def:strong.0}.
\end{proof}
 
To infer that sufficiently regular energy-variational solutions
are already strong solutions, 
we assume that $\potential$ is given in integral form
$\potential(\Se)=\int_\Om\mathfrak{P}(\Se)\,\ddd x$.
This leads to approximation properties 
that are sufficient to 
show that energy-variational solutions
with more regularity, that is, belonging to $D_{\mathcal K}\cap\Ybzero$,
are already strong solutions.
For simplicity, we confine ourselves to energy-variational solutions of type $\mathcal K_s$
given in \eqref{K2}. Note that, by the monotonicity property in Proposition~\ref{prop:type.monotonous}, this includes the (smaller) set of solutions of type $\mathcal{K}(\tS)=C\bp{\|\tS\|_{L^\infty(\Omega)}^2+\|\nabla \tS\|_{L^3(\Omega)}^2}\le \mathcal K_s(\tvvel,\tS)$, whose existence follows from Theorem~\ref{thm:envar}.

\begin{proposition}[Regular energy-variational solutions]\label{prop:addreg}
	Let $(\f v, \Se)\in \Xbzero$ be an energy-vari\-a\-tion\-al solution of type $\mathcal{K}$ according to Definition~\ref{def:envar} with $\mathcal K= \mathcal{K}_s$ given in~\eqref{K2},
	which satisfies $(\f v,\Se)\in D_\mathcal{K}\cap\Ybzero$. 
	Further suppose that the associated initial data $(\vvel_0,\Se_0)$ satisfy $\potential(\Se_0)<\infty$,
	and that $\potential$ is in integral form, that is,
	 $\potential(\Se)=\int_\Om\mathfrak{P}(\Se)\,\ddd x$
	for a lower semicontinuous, convex function $\mathfrak{P}:\R^{3\times 3}_{\delta}\to[0,\infty]$ with $\mathfrak{P}(\zerotens)=0$.
	Then $(\f v ,\Se)$ is a strong solution in the sense of Definition~\ref{def:strong.0} taking the same initial data $(\vvel_0,\Se_0)$.
\end{proposition}

\begin{proof}
First, we show that the initial data are attained.
To this end, we choose a sequence of regular functions $\{(\vvel_{0,\ve},\Se_{0,\ve})\}\subset \CRcisigma\np{\Omega}\times C^\infty(\overline\Om;\R^{3\times 3}_{\delta})$ converging to  
$(\vvel_{0},\Se_{0})$ in $\LRsigma{2}(\Om)\times\LRdev{2}(\Om)$ 
in such a way that $\potential(\Se_{0,\ve})\to\potential(\Se_0)$. A sequence $\{\Se_{0,\ve}\}$ with these properties can be obtained by mollification of the function $\Se_0$ (extended by zero on $\mathbb{R}^3\setminus\Om$). 
Indeed, note that for a standard mollifying kernel $(\rho_\varepsilon)_{\ve\in(0,1)}\subset C^\infty_0(\mathbb{R}^3)$ with $\rho_\varepsilon\ge0$ and $\int_{\mathbb{R}^3}\rho_\varepsilon=1$,  Jensen's inequality ensures that the function $\Stens_{0,\varepsilon}$ defined by $\Stens_{0,\varepsilon}(x)\coloneqq\rho_\varepsilon\ast\Se_0(x):=\int_{\mathbb{R}^3}\rho_\varepsilon(y)\Se_0(x-y)\,\ddd y$ satisfies
the pointwise inequality
\[
	\mathfrak{P}(\rho_\varepsilon\ast\Se_0)\le \rho_\varepsilon\ast\mathfrak{P}(\Se_0)\qquad\text{in }\Om,
\]
and hence
\[
\limsup_{\varepsilon\to0}\potential(\Stens_{0,\varepsilon})
\le \limsup_{\varepsilon\to0}
\int_\Om\rho_\varepsilon\ast\mathfrak{P}(\Se_0)\,\ddd x
=\int_\Om\mathfrak{P}(\Se_0)\,\ddd x
=\potential(\Stens_0).
\]
Since $\Stens_{0,\varepsilon}\to\Stens_0$ in $\LRdev{2}(\Omega)$
and $\potential$ is lower semicontinuous,
this even implies that $\potential(\Stens_{0,\varepsilon})\to\potential(\Stens_0)$ as $\varepsilon\to0$.
\EEE
Choosing $(\tvvel,\tS)$ constant in time and equal to $(\vvel_{0,\ve},\Se_{0,\ve})$ yields test functions admissible in the relative energy inequality \eqref{relen.}. It now suffices to pass to the limit $t\to0$ in this inequality along an admissible sequence of times, which gives (thanks to $\Ybzero\subset C([0,T);\LRsigma{2}(\Om)\times\LRdev{2}(\Om))$) the estimate
\begin{equation}
	\mathcal R(\vvel(0),\Stens(0)|\vvel_{0,\ve},\Se_{0,\ve})
	\leq \mathcal R(\vvel_0,\Stens_0|\vvel_{0,\ve},\Se_{0,\ve}).
\end{equation}
Taking the limit $\ve\to0$, we infer that $\mathcal R(\vvel(0),\Stens(0)|\vvel_{0},\Se_{0})=0$,
whence $\vvel(0)=\vvel_0$ and $\Stens(0)=\Stens_0$.

In the next step we show the weak formulation \eqref{eq:weaksol.v}.
Due to $\vvel(0)=\vvel_0$, this is equivalent to showing the identity
\begin{equation}\label{strongnavier}
\int_0^T\left [ \langle \t \f v, \Phi \rangle  + \int_\Omega ( \f v \cdot \nabla) \f v  \cdot \Phi +( \eta \Se + \mu \nabla \f v ) : \nabla \Phi \dx\right ] \dt = \int_0^T \langle \f f , \Phi \rangle \dt 
\end{equation}
for all $\Phi\in\CRcisigma\np{\OmTzero}$.
Since $(\f v,\Se)\in D_\mathcal{K}\cap\Ybzero$, we can use 
$(\vv,\tS) = (\f v + \alpha \f u, \Se + \alpha \mathbb{T})$ with $\alpha \in (0,1)$ and $(\f u , \mathbb{T}) \in D_\mathcal{K}\cap\Ybzero$
as a test function in \eqref{relen.}. 
Due to the convexity of $\mathcal{P}$, for $\alpha \in (0,1)$ we may estimate 
 \begin{align}\label{convin}
 \begin{split}
 \mathcal{P}(\Se) -  \mathcal{P}( \Se + \alpha \mathbb{T} ) ={}& \mathcal{P}(\Se) - \mathcal{P}\left ((1-\alpha) \Se  + \alpha ( \Se + \mathbb{T})\right ) 
\\
\geq {}&  \mathcal{P}(\Se) - (1-\alpha) \mathcal{P} (\Se ) - \alpha \mathcal{P} ( \Se + \mathbb{T}) \\
={}& \alpha \left (  \mathcal{P} ( \Se )  -\mathcal{P} (\Se +  \mathbb{T})\right ) \,.
 \end{split}
 \end{align}
Inserting this into~\eqref{relen.}, multiplying the resulting relation by $1/\alpha $, and sorting the different terms according to the appearing exponent of $\alpha$, we end up  with
\begin{align}
\alpha R( \f v , \Se , \f u , \mathbb{T}, \alpha ) + \int_0^t\left [ \mathcal{P} ( \Se )  -\mathcal{P} (\Se +  \mathbb{T}) 
- \left \langle \mathcal{A}_0(\f v ,\Se) , \begin{pmatrix}
\f u \\ \mathbb{T}
\end{pmatrix}\right \rangle\right ] \e^{\int_s^t \mathcal{K}(\f v +\alpha \f u , \Se +\alpha \mathbb{T})\de \tau} \de s \leq 0 \,.\label{alphaineq}
\end{align}
The remainder term $R$ contains all terms with prefactor $\alpha$ and is given by 
\begin{align*}
 R( \f v , \Se , \f u , \mathbb{T}, \alpha ):={}&  \frac{1}{2}\left ( \| \f u(t) \| _{L^2(\Omega)}^2 + \| \mathbb{T}(t) \|_{L^2(\Omega)}^2 \right ) + \int_0^t R_2( \f v , \Se , \f u , \mathbb{T}, \alpha ) \e^{\int_s^t\mathcal{K}(\f v + \alpha \f u , \Se+\alpha \mathbb{T}) \de \tau } \de s 
\\& - \frac{1}{2} \left ( \| \f u(0) \|_{L^2(\Omega)}^2 + \| \mathbb{T}(0)\|_{L^2(\Omega)}^2 \right ) \e^{\int_0^t \mathcal{K}(\f v + \alpha \f u , \Se+\alpha \mathbb{T})  \de s }\,,
\end{align*}
where $R_2 $ is given by
\begin{align*}
 R_2( \f v , \Se , \f u , \mathbb{T}, \alpha ):={}&  \mu \| \nabla \f u \|_{L^2(\Omega)}^2 - \int_\Omega (\f u \cdot \nabla)\f u \cdot (\f v + \alpha \f u)-\mathbb{T}\otimes \f u \dreidots \nabla ( \Se + \alpha \mathbb{T})  \dx 
 \\&-\int_\Omega \left ( \mathbb{T}(\nabla \f u )_{\skw} - (\nabla \f u) _{\skw} \mathbb{T}\right ) : ( \Se + \alpha \mathbb{T}) \dx 
 \\
& -
\int_\Omega  \left (\t \f u  + (\f u \cdot \nabla ) (\f v + \alpha \f u ) + (\f v \cdot \nabla) \f u 
\right )\cdot \f u   + \left ( \eta \mathbb{T} + 2 \mu (\nabla \f u)_{\sym} \right ) : \nabla \f u 
 \dx
 \\&-
 \int_\Omega\left ( \t \mathbb{T}+  (\f u \cdot \nabla ) (\Se + \alpha \mathbb{T} ) + (\f v \cdot \nabla) \mathbb{T} +  \mathbb{T}  (\nabla \f v )_{\skw} -  (\nabla \f v )_{\skw} \mathbb{T}  \right ) : \mathbb{T}\dx
\\&-
\int_\Omega\left ( (\Se + \alpha \mathbb{T} ) (\nabla \f u )_{\skw} -  (\nabla \f u )_{\skw} (\Se + \alpha \mathbb{T} )  - \eta (\nabla \f u )_{\sym} \right ) : \mathbb{T}\dx
\\ &+ \frac{1}{2}\mathcal{K} (\f v +\alpha \f u , \Se +\alpha \mathbb{T}) \left ( \| \f u \|_{L^2(\Omega)}^2 + \| \mathbb{T}\|_{L^2(\Omega)}^2 \right ) 
\,.
\end{align*}
Passing to the limit $\alpha \ra 0$ in~\eqref{alphaineq}, we infer from the boundedness of all terms in $R$ and the continuity property~$\lim _{\alpha \ra 0} \mathcal{K}(\f v + \alpha \f u , \Se + \alpha \mathbb{T}) = \mathcal{K}(\f v , \Se)$, which holds for $\mathcal{K}=\mathcal{K}_s$ given in~\eqref{K2},  that 
\begin{align}
\int_0^t\left [ \mathcal{P} ( \Se )  -\mathcal{P} (\Se +  \mathbb{T}) 
- \left \langle \mathcal{A}_0(\f v ,\Se) , \begin{pmatrix}
\f u \\ \mathbb{T}
\end{pmatrix}\right \rangle\right ] \e^{\int_s^t\mathcal{K}(\f v , \Se)\de \tau} \de s \leq 0 \,.\label{ineqreg}
\end{align}
Choosing  $\mathbb{T}=\zerotens$ and $ \f u =  \Phi \e^{-\int_\cdot^t \mathcal{K}(\f v , \Se)\de \tau}$ for $\Phi\in\CRcisigma\np{\OmTzero}$,
 we first infer the weak formulation~\eqref{strongnavier}  with an inequality.  Choosing $ \f u = - \Phi \e^{-\int_\cdot ^t \mathcal{K}(\f v , \Se)\de \tau}$ implies the converse inequality.
In summary, equation~\eqref{strongnavier} is fulfilled for all $\Phi\in\CRcisigma\np{\OmTzero}$. 

Reinserting this information into~\eqref{ineqreg}, we find 
\begin{equation}\label{insub}
\int_0^t \left [\mathcal{P}(\Se ) - \mathcal{P}(\Se + \mathbb{T}) 
+ \int_\Omega \mathbb{G}  : \mathbb{T}\dx
 \right ] 
\e^{\int_s^t\mathcal{K}(\f v  , \Se)\de \tau}  \de s \leq 0 \,
\end{equation}
with 
\[
 \mathbb{G} \coloneqq- \left (  \t \Se + ( \f v \cdot \nabla) \Se + \left ( \Se (\nabla \f v )_{\skw} - ( \nabla \f v )_{\skw} \Se \right ) - \eta (\nabla \f v )_{\sym} \right )\,,
\]
where the regularity of $(\f v , \Se)\in D_{\mathcal{K}}\cap  \Ybzero  $ guarantees that $\mathbb{G} \in L^1(0,t;\LRdev{2}(\Omega))$. 
By choosing $ \mathbb{T}= \tilde {\mathbb{T} }\e^{-\int_\cdot^t \mathcal{K}(\f v , \Se)\de \tau} $ 
for $(\zerovec,\tilde{\mathbb{T}})\in D_{\mathcal{K}}\cap  \Ybzero $
and employing inequality~\eqref{convin}, we infer from~\eqref{insub} that
\begin{equation*}
\int_0^t \bb{\mathcal{P}(\Se +  \tilde {\mathbb{T} }) -\mathcal{P}(\Se )}\,\ds 
\geq \int_0^t\int_\Omega \mathbb{G}  :  \tilde {\mathbb{T} }\dx
\de s  \,.
\end{equation*}
We now choose $\tTtens(x,t)=\varphi(t)(\hat\Ttens(x) - \Se(x,t))$ with 
$\varphi\in\CRci((0,T))$, $0\leq\varphi\leq 1$, and 
$\hat\Ttens\in \CRi(\overline\Om;\R^{3\times 3}_{\delta})$.
Then $(\zerovec,\tilde{\mathbb{T}})\in D_{\mathcal{K}}\cap  \Ybzero $,
and 
similarly to \eqref{convin}, we conclude
\[
\int_0^t \varphi\bb{\potential(\hat\Ttens)-\potential(\Stens)}\,\ds
\geq \int_0^t\int_\Omega\varphi \mathbb{G}:(\hat\Ttens- \Se )\,\dx\ds.
\]
By linearity, this inequality holds for all $\varphi\in\CRci((0,T))$ with $\varphi\geq 0$,
which implies that
\begin{equation}
\potential(\hat\Ttens)-\potential(\Stens(t))
\geq \int_\Omega\mathbb{G}:(\hat\Ttens - \Stens(t)) \,\dx
\label{pointwiseT}
\end{equation}
for a.a.~$t\in(0,T)$.
It is not difficult to see  that the null set where \eqref{pointwiseT} may not be valid can be chosen independently of 
$\hat\Ttens\in\CRi(\overline\Om;\R^{3\times 3}_{\delta})$.
Now fix $t\in(0,T)$ such that \eqref{pointwiseT} holds.
If $\hat\Ttens\in\LRdev{2}(\Omega)$ with $\potential(\hat\Ttens)=\infty$,
then \eqref{pointwiseT} is trivially satisfied. 
If $\hat\Ttens\in\LRdev{2}(\Omega)$ with $\potential(\hat\Ttens)<\infty$,
then we can use the same mollifier argument as above
to obtain a sequence $(\hat\Ttens_j)\subset\CRi(\overline\Om;\R^{3\times 3}_{\delta})$
such that $\hat\Ttens_j\to\hat\Ttens$ in $\LRdev{2}(\Omega)$
and $\potential(\hat\Ttens_j)\to\potential(\hat\Ttens)$ as $j\to\infty$.
For this sequence, \eqref{pointwiseT} is satisfied,
and a passage to the limit $j\to\infty$ shows that $\hat\Ttens$ also satisfies \eqref{pointwiseT}.
In conclusion, \eqref{pointwiseT} holds for all $\hat\Ttens\in\LRdev{2}(\Omega)$,
and identity \eqref{eq:strongvar.0} follows for a.a.~$t\in(0,T)$
by definition of the subdifferential $\partial\potential$.
\end{proof}

 \appendix
 
 \section{Appendix: Existence of generalized solutions}
 \label{sec:appendix}

The purpose of this section is to show Theorem~\ref{thm:gensol}, \textit{i.e.}, the existence of generalized solutions in the sense of Definition~\ref{def:gensol}.

\begin{proof}[Proof of Theorem~\ref{thm:gensol}] 
In \cite{EiterHopfMielke2021} existence of a solution to~\eqref{eqvis}
was established using a notion of generalized solution similar to Definition~\ref{def:gensol}, but with the space of test functions $\bigcup_{2<q<\infty}\ZT^q$ replaced with the smaller space 
\begin{align}
	\HSR{1}(0,t;\LR{2}_\delta(\Omega))
\cap \LR{2}(0,t;\HSR{1}(\Omega)^{3\times3})\cap 
\LR{5}(0,t;\LR{5}(\Om)^{3\times3})
\end{align}
for any $t\in(0,T)$, and inequality~\eqref{est:varin0T} replaced with its weaker form
\begin{equation}\label{est:varin0}
		\begin{aligned}
			\int_0^{t}\!\langle\pt\tStens,\Stens{-}\tStens\rangle
			+\potential(\Stens)-\potential(\tStens)\ds
			+\int_0^{t}\!\!\!\int_\Omega\gamma\grad\Stens:\grad(\Stens{-}\tStens)&\dx\ds 
			\\
			-\int_0^{t}\!\int_\Omega \bp{\vvel\cdot\grad\Stens
				+\Stens\skewgrad\vvel-\skewgrad\vvel\Stens}:\tStens 
			&-\eta\symmgrad\vvel:(\Stens-\tStens)\dx\ds\\
			&\qquad\quad\leq\tfrac{1}{2}\|\Stens_0-\tStens(0)\|_{L^2(\Omega)}^2.
		\end{aligned}
	\end{equation}
	However, going through the proof of \cite[\ThmMain]{EiterHopfMielke2021},
	one easily verifies that the constructed solution satisfies \eqref{est:varin0}
	even for all $\tStens$ in the larger class $\ZT^q$, $q\in(2,\infty)$. Regarding inequality~\eqref{est:varin0T}, we will show in Lemma~\ref{l:0=>0T} below that it can directly be derived from~\eqref{est:varin0}.

	It is necessary to mention that inequality~\eqref{est:enineq.v} for $\np{\vvel,\Stens}$ has not been stated explicitly  
	in \cite{EiterHopfMielke2021},
	but only for an approximating family $(\vvel_\varepsilon,\Stens_\varepsilon)$ 
	that satisfies
	\[
	\begin{aligned}
		\vvel_\varepsilon&\rightharpoonup\vvel && \tin \LR{2}(0,t;\HSR{1}(\Omega)^3),
		\\
		\vvel_\varepsilon&\overset{\ast}{\rightharpoonup}\vvel && \tin \LR{\infty}(0,t;\LRsigma{2}(\Omega)),
		\\
		\vvel_\varepsilon&\to\vvel && \tin \LR{2}(0,t;\LR{2}(\Omega)^3),
		\\
		\Stens_\varepsilon&\rightharpoonup\Stens && \tin \LR{2}(0,t;\HSR{1}(\Omega)^{3\times3})
	\end{aligned}
	\]
	as $\varepsilon\to0$.
	In particular, 
	the strong convergence statement implies $\norm{\vvel_\varepsilon(t)}_{L^2(\Omega)}\to\norm{\vvel(t)}_{L^2(\Omega)}$
	for almost every $t\in(0,T)$, at least for a suitable subsequence.
	This allows to pass to the limit $\varepsilon\to0$ in \eqref{est:enineq.v}
	(with $\np{\vvel,\Stens}$ replaced with $\np{\vvel_\varepsilon,\Stens_\varepsilon}$),
	where for the last term the identity
	\[
	\int_\Omega\eta\Stens:\grad\vvel \dx
	=-\int_\Omega\eta(\Div\Stens)\cdot\vvel \dx
	\]
	can be used. 
	In the end, $\np{\vvel,\Stens}$ also satisfies \eqref{est:enineq.v}.
	
	Finally, observe that in \cite{EiterHopfMielke2021}
	the boundary $\partial\Omega$ was assumed to be of class $\CR{1,1}$,
	which allowed for the construction of a suitable extension 
	of the boundary conditions.
	Since we consider homogeneous boundary conditions \eqref{eq:bdrycond} here,
	we can use the trivial extension, whence Lipschitz boundary is sufficient.
\end{proof}

The following lemma shows that inequality~\eqref{est:varin0}
implies the stronger inequality~\eqref{est:varin0T},
which also takes into account the values of $\Stens$ and $\tStens$ at time $t$.

\begin{lemma}\label{l:0=>0T}
	Let $q\in(2,\infty)$.
	If $(\vvel,\Stens)\in\Xb$ satisfies~\eqref{est:varin0} for all $t\in(0,T)$ and all $\tStens \in \ZT^q$, 
	then~\eqref{est:varin0T} holds true
	for almost all $t\in(0,T)$ and all $\tStens \in \ZT^q$.
\end{lemma}

\begin{proof}
	The assertion is obtained similarly as in the proof of~\cite[\PropEnIneq]{EiterHopfMielke2021}.
	We therefore only sketch the argument. 
	Extending $\Stens$ by zero for $t<0$, we define for $\kappa>0$
	\[
	\Stens_\kappa(s)=\kappa^{-1}\int_{s-\kappa}^s\Stens(\tau)\dtau.
	\]
	Fix $t\in(0,T)$ and let $\tStens_{\kappa,\delta}:=\np{1-\zeta_\delta}\tStens+\zeta_\delta \Stens_\kappa$, where $\zeta_\delta(s):=\zeta((s-t)/\delta)$
	for some non-decreasing function $\zeta\in \CRi(\mathbb{R};[0,1])$ satisfying
	$\zeta(s)=0$ for $s\le-1$ and $\zeta(s)=1$ for $s\ge0$,
	and $\tStens\in\ZT^q$ is arbitrary. 
	We observe that $\tStens_{\kappa,\delta}\in\ZT^q$ satisfies
	$\tStens_{\kappa,\delta}(0)=\tStens(0)$, $\tStens_{\kappa,\delta}(t)=\Stens_\kappa(t)$
	and, as $\kappa,\delta\downarrow0$, the sequence $\{\tStens_{\kappa,\delta}\}$ approximates $\tStens$ (in a suitable sense).
	To infer ineq.~\eqref{est:varin0T}, we insert $\tStens_{\kappa,\delta}$ as a test function in ineq.~\eqref{est:varin0}, take the limit $\delta\downarrow0$ and then send $\kappa\downarrow0$.
	Let us only point out how to pass to these limits in the terms involving a time derivative, since the remaining integrals can be handled as in~\cite[\PropEnIneq]{EiterHopfMielke2021}.
	We compute
\begin{multline*}		\int_0^{t}\!\langle\pt\tStens_{\kappa,\delta},\tStens_{\kappa,\delta}\rangle\ds
		=  \frac{1}{2}\norm{\Stens_\kappa(t)}_{L^2(\Omega)}^2 - \frac{1}{2}\norm{\tStens(0)}_{L^2(\Omega)}^2
		\\
		= \frac{1}{2}\norm{\Stens_\kappa(t)}_{L^2(\Omega)}^2 - \frac{1}{2}\norm{\tStens(t)}_{L^2(\Omega)}^2
		+\int_0^{t}\!\langle \pt\tStens,\tStens\rangle\ds
\end{multline*}	and
	\[
	\begin{aligned}
		\int_0^{t}\!\langle \pt\tStens_{\kappa,\delta},\Stens\rangle\ds
		= 
		&\int_0^{t}\!-\zeta_\delta'(s)\int_\Omega \tStens:\Stens\dx\ds 
		+\int_0^{t}\!\zeta_\delta'(s)\int_\Omega \Stens_\kappa:\Stens\dx\ds 
		\\
		&\quad+\int_{0}^{t} \!\np{1-\zeta_\delta(s)} \!
		\langle\pt \tStens,\Stens\rangle\ds
		+\int_{t-\delta}^{t} \!\zeta_\delta(s)\!\langle\pt \Stens_\kappa,\Stens\rangle\ds,
	\end{aligned}
	\]
	where the last term in the second equation vanishes as $\delta\to0$.
	Thus, combining these two identities
	and treating the remaining terms as in the proof of 
	\cite[\PropEnIneq]{EiterHopfMielke2021}, 
	we deduce for a.a.~$t\in(0,T)$ and all $\tStens\in\ZT^q$ that
	\[
	\begin{aligned}
		\lim_{\kappa\to0}&\;\lim_{\delta\to0}\;
		\int_0^{t}\!\langle
		\pt\tStens_{\kappa,\delta},\tStens_{\kappa,\delta}-\Stens\rangle\ds
		\\
		&= 
		\frac{1}{2} \norm{\Stens(t)}_{L^2(\Omega)}^2 
		- \frac{1}{2}\norm{\tStens(t)}_{L^2(\Omega)}^2
		+\int_0^{t}\!\langle\pt\tStens,\tStens\rangle\ds
		-\int_\Omega \!\bp{\Stens(t)-\tStens(t)}:\Stens(t)\dx
		-\int_0^{t}\!\langle\pt\tStens,\Stens\rangle\ds
		\\
		&=\int_0^{t}\!\langle\pt\tStens,\tStens-\Stens\rangle\ds
		-\frac{1}{2}\norm{\tStens(t)-\Stens(t)}_{L^2(\Omega)}^2.
	\end{aligned}
	\]
	Proceeding with the remaining terms as in 
	\cite[\PropEnIneq]{EiterHopfMielke2021}, we arrive at
	\eqref{est:varin0T}.
\end{proof}

%%%%%%%%%%%%%%%%%%%%%%%%%%%%%%%%%%%%%%%%%%%%%%%%%%%%%%%%%%%%%%
%%          Bibliography                                    %%
%%%%%%%%%%%%%%%%%%%%%%%%%%%%%%%%%%%%%%%%%%%%%%%%%%%%%%%%%%%%%%
\small
\bibliographystyle{abbrv}

\end{document}